\documentclass[a4paper,oneside,10pt]{article}%
\usepackage{amsmath}
\usepackage{amsfonts}
\usepackage{amssymb}
\usepackage{graphicx}
\usepackage[square,numbers,sort&compress]{natbib}%
\usepackage[colorlinks,linkcolor=blue,anchorcolor=blue,citecolor=blue]{hyperref}
\providecommand{\U}[1]{\protect \rule{.1in}{.1in}}

\newtheorem{theorem}{Theorem}[section]

\newtheorem{corollary}[theorem]{Corollary}

\newtheorem{definition}[theorem]{Definition}

\newtheorem{lemma}[theorem]{Lemma}

\newtheorem{proposition}[theorem]{Proposition}
\newtheorem{remark}[theorem]{Remark}

\newenvironment{proof}[1][Proof]{\noindent \textbf{#1.} }{\  \rule{0.5em}{0.5em}}
\numberwithin{equation}{section}

\begin{document}

\title{On the strong Markov property for stochastic differential equations driven by $G$-Brownian motion}
\author{Mingshang Hu \thanks{Zhongtai Securities Institute for Financial Studies, Shandong University, humingshang@sdu.edu.cn. Research supported by NSF (No. 11671231) and Young Scholars Program of Shandong University (No. 2016WLJH10)} \and Xiaojun Ji \thanks{School of Mathematics, Shandong University, xiaojunji@163.com} \and Guomin Liu \thanks{Zhongtai Securities Institute for Financial Studies, Shandong University, gmliusdu@163.com (Corresponding author). Research supported by NSF (11601282) and Shandong Province NSF (No. ZR2016AQ10).
Hu, Ji and Liu's  research
was partially supported by NSF (No. 11526205 and
No. 11626247) and the 111 Project (No. B12023). }}
\date{}
\maketitle

\textbf{Abstract}. The objective of this paper is to study the strong Markov property for the
stochastic differential equations driven by $G$-Brownian motion ($G$-SDEs for short).  We first extend the deterministic-time conditional $G$-expectation to optional times. The strong Markov property for $G$-SDEs is then obtained by Kolmogorov's criterion for tightness. In particular,
for any given optional time $\tau$ and $G$-Brownian motion $B$, the reflection
principle for $B$ holds and $(B_{\tau+t}-B_{\tau})_{t\geq0}$ is still a
$G$-Brownian motion.


\addcontentsline{toc}{section}{\hspace*{1.8em}Abstract}

\section{Introduction}
The  strong Markov property for stochastic differential equations (SDEs)  is  one of the most fundamental results in the theory of classical stochastic processes. It
claims that for any given optional time $\tau$ we have
\begin{equation}\label{strong markov for classical sde}
{{E}}_P[\varphi(X_{\tau+t_1}^x,\cdots,X_{\tau+t_m}^x)|\mathcal{F}_{{\tau+}}]
={{E}}_P[\varphi(X_{t_1}^y,\cdots,X_{t_m}^y)]_{y=X_\tau^x}
\end{equation}
for SDEs $(X^x_t)_{t\geq 0}$ with initial value $x$. Here ${{E}}_P$ and ${{E}}_P[\cdot|\mathcal{F}_{{\tau+}}]$ stands for the expectation and conditional expectation, respectively, related to a probability measure $P$. It was obtained by K. It\^{o} in his pioneering work \cite{It0}, and since then,  it has been widely applied to
 stochastic control,  mathematical finance and probabilistic method for
partial differential equations (PDEs); see,  e.g., \cite{BL,Fr,Ok}.

Recently, motivated by probabilistic interpretations for fully nonlinear PDEs and financial problems with model uncertainty,  Peng \cite{P3,P4,P7}  systematically introduced the notion of nonlinear $G$-expectation $\hat{\mathbb{E}}[\cdot]$   by stochastic control and
PDE methods. Under the $G$-expectation framework, a new kind of Brownian motion, called $G$-Brownian motion, was constructed. The corresponding stochastic calculus
of It\^{o}'s type was also established.  Furthermore, by  the contracting mapping theorem, Peng  obtained the existence and
uniqueness of the solution of $G$-SDEs:
\begin{equation}%
\begin{cases}
dX_{t}^{x}=b(X_{t}^{x})dt+\sum_{i,j=1}^{d}h_{ij}(X_{t}^{x})d\langle
B^{i},B^{j}\rangle_{t}+\sum_{j=1}^{d}\sigma_{j}(X_{t}^{x})dB_{t}%
^{j},\  \  \  \ t\in \lbrack0,T],\\
X_{0}^{x}=x,
\end{cases}
\label{GSDE in intro}%
\end{equation}
where  $B=(B^{1},\ldots,B^{d})$ is $G$-Brownian motion and  $\langle B^{i},B^{j}\rangle$ is its cross-variation process, which is  not deterministic unlike the classical case.

A very interesting problem is whether, for $G$-SDEs, the following
generalized  strong Markov property is true:
\begin{equation}
\hat{\mathbb{E}}_{\tau+}[\varphi(X_{\tau+t_{1}}^{x},\cdots,X_{\tau+t_{m}}%
^{x})]=\hat{\mathbb{E}}[\varphi(X_{t_{1}}^{y},\cdots,X_{t_{m}}^{y}%
)]_{y=X_{\tau}^{x}}.\label{Strongmar}%
\end{equation}
In this paper, we first  construct the conditional $G$-expectation $\hat{\mathbb{E}}_{\tau+}[\cdot]$ for any given optional time $\tau$ by extending  the definition of  conditional $G$-expectation $\hat{\mathbb{E}}_t[\cdot]$ to   optional times.  The main tools in this construction are  a  universal continuity estimate   for
$\hat{\mathbb{E}}_{t}[\cdot]$ (see Lemma \ref{Et continuity lemma}) and a new
kind of consistency property (see Proposition \ref{main proposition}). We also show that $\hat{\mathbb{E}}_{\tau+}[\cdot]$ can preserve most useful properties of classical conditional expectations except the linearity.
Based on the conditional expectation $\hat{\mathbb{E}}_{\tau+}[\cdot]$, we then further obtain the strong Markov property (\ref{Strongmar}) for
$G$-SDEs by adapting the standard discretization method. In contrast to the linear case, the main difficulty is that  in the nonlinear expectation context the dominated convergence theorem does not hold in general. We tackle this problem  by using
Kolmogorov's criterion for tightness and the properties of $\hat{\mathbb{E}%
}_{\tau+}[\cdot]$.
In particular, for  $G$-Brownian motion $B$, we obtain
that the reflection principle for $B$  holds and $(B_{\tau+t}-B_{\tau
})_{t\geq0}$ is still a $G$-Brownian motion. Finally, with the help of the strong Markov
property, the level set of $G$-Brownian motion is also investigated.

We note that problem of constructing  $\hat{\mathbb{E}}_{\tau+}[\cdot]$ was first considered in   \cite{NH}, where $\hat{\mathbb{E}}_{\tau+}[\cdot]$ is defined for all upper semianalytic (more general than Borel-measurable) functions by the analytic sets theory. But the corresponding conditional expectation is also upper semianalytic and when the usual Borel-measurablity can be attained remains unknown. In our paper, by a   completely different approach, our construction focuses on a large class of Borel functions to obtain more regularity properties for $\hat{\mathbb{E}}_{\tau+}[\cdot]$, among which is its measurability  with respect to $\mathcal{F}_{\tau+}$. Moreover, some of these  properties are important for the derivation of strong Markov property for $G$-SDEs.

This paper is organized as follows. In Section 2, we recall some basic notions
of $G$-expectation, $G$-Brownian motion and $G$-SDEs. Section 3 is devoted to
the construction of the conditional $G$-expectation $\hat{\mathbb{E}}_{\tau+}[\cdot]$
and the investigation of its properties. Then, in Section 4, we study the strong
Markov property for $G$-SDEs. Finally, in Section 5, we use the strong Markov
property to prove that the level set of $G$-Brownian motion has no isolated point.

\section{Preliminaries}
In this section, we review some basic notions and results of $G$-expectation. More relevant details can be found in \cite{GJ,Linq,Liny,LW,P3,P4,P7,P9}
\subsection{$G$-expectation space}
Let $\Omega$ be a given nonempty set and $\mathcal{H}$ be a linear space of
real-valued functions on $\Omega$ such that if $X_{1}$,$\dots$,$X_{d}%
\in \mathcal{H}$, then $\varphi(X_{1},X_{2},\dots,X_{d})\in \mathcal{H}$ for
each $\varphi \in C_{b.Lip}(\mathbb{R}^{d})$, where $C_{b.Lip}(\mathbb{R}^{d})$ is the space of bounded, Lipschitz functions on $\mathbb{R}^{d}$.
$\mathcal{H}$ is considered as the space of random variables.
\begin{definition}
A sublinear expectation $\hat{\mathbb{E}}$ on $\mathcal{H}$ is a functional
$\mathbb{\hat{E}}:\mathcal{H}\rightarrow \mathbb{R}$ satisfying the following
properties: for each $X,Y\in \mathcal{H}$,
\begin{description}
\item[{\rm (i)}] {Monotonicity:}\quad$\mathbb{\hat{E}}[X]\geq \mathbb{\hat{E}%
}[Y]\  \  \text{if}\ X\geq Y$;
\item[{\rm (ii)}] {Constant preserving:}\quad$\mathbb{\hat{E}}%
[c]=c\  \  \  \text{for}\ c\in \mathbb{R}$;
\item[{\rm (iii)}] {Sub-additivity:}\quad$\mathbb{\hat{E}}[X+Y]\leq
\mathbb{\hat{E}}[X]+\mathbb{\hat{E}}[Y]$;
\item[{\rm (iv)}] {Positive homogeneity:}\quad$\mathbb{\hat{E}}[\lambda
X]=\lambda \mathbb{\hat{E}}[X]\  \  \  \text{for}\  \lambda \geq0$.
\end{description}
The triple $(\Omega,\mathcal{H},\mathbb{\hat{E}})$ is called a sublinear
expectation space.
\end{definition}
\begin{definition}
Two $d$-dimensional random vectors $X_{1}$ and $X_{2}$ defined respectively on
sublinear expectation spaces $(\Omega_{1},\mathcal{H}_{1},\mathbb{\hat{E}}%
_{1})$ and $(\Omega_{2},\mathcal{H}_{2},\mathbb{\hat{E}}_{2})$ are
called identically distributed, denoted by $X_{1}\overset{d}{=}X_{2}$, if%
\[
\mathbb{\hat{E}}_{1}[\varphi(X_{1})]=\mathbb{\hat{E}}_{2}[\varphi
(X_{2})], \ \ \ \ \text{for each} \ \varphi \in C_{b.Lip}(\mathbb{R}^{d}).
\]
\end{definition}

\begin{definition}
On the sublinear expectation space $(\Omega,\mathcal{H},\hat{\mathbb{E}})$, an $n$-dimensional random vector $Y$ is said to be independent from a $d$-dimensional random vector $X$, denoted by $Y\bot X$, if
$$\hat{\mathbb{E}}[\varphi(X,Y)]=\hat{\mathbb{E}}[\hat{\mathbb{E}}[\varphi(x,Y)]_{x=X}], \ \ \ \ \text{for each}\ \varphi\in C_{b.Lip}(\mathbb{R}^{d+n}).$$
\end{definition}
A $d$-dimensional random vector $\bar{X}$ is said to be an independent copy of  $X$ if $\bar{X}\overset{d}{=} X$ and $\bar{X}\bot X$.
\begin{definition}\textbf{($G$-normal distribution)}
A $d$-dimensional random vector $X$ defined on $(\Omega,\mathcal{H},\hat{\mathbb{E}})$ is called $G$-normally distributed if for any $a,b\geq0$,
$$aX+b\bar{X}\overset{d}{=}\sqrt{a^2+b^2}X,$$
where $\bar{X}$ is an independent copy of  $X$. Here the letter $G$ denotes the function $G(A):=\frac12 \hat{\mathbb{E}}[\langle AX,X\rangle]$ for $A\in \mathbb{S}(d)$, where $\mathbb{S}(d)$ denotes the
space of all $d \times d$ symmetric matrices.
\end{definition}

In the rest of this paper, we denote by $\Omega:=C([0,\infty); \mathbb{R}^d)$ the space of all $\mathbb{R}^d$-valued continuous paths $(\omega_t)_{t\geq0}$, equipped with the distance
$$\rho_d(\omega^1,\omega^2):=\sum_{i=1}^\infty\frac1{2^i}[(||\omega^1-\omega^2||_{C^d[0,i]}\wedge 1)],$$
where $||\omega^1-\omega^2||_{C^d[0,T]}:=\max_{t\in[0,T]}|\omega_t^1-\omega_t^2|$ for $T>0$.
Given any $T>0$, we also define $\Omega_T:=\{(\omega_{t\wedge T})_{t\geq 0}:\omega\in\Omega\}$.

Let $B_t(\omega):=\omega_t$ for $\omega\in \Omega$, $t\geq 0$ be the canonical process. We set
$$L_{ip}(\Omega_T):=\{\varphi(B_{t_1},B_{t_2}-B_{t_1}\cdots,B_{t_n}-B_{t_{n-1}}):n\in\mathbb{N},0\leq t_1<t_2\cdots<t_n\leq T,\varphi\in C_{b.Lip}(\mathbb{R}^{d\times n})\}$$
as well as
\begin{equation}\label{9237257894334}
L_{ip}(\Omega):=\bigcup_{m=1}^ \infty L_{ip}(\Omega_m).
\end{equation}

Let $G:\mathbb{S}(d)\rightarrow\mathbb{R}$ be a given monotonic and sublinear function. The $G$-expectation  on $L_{ip}(\Omega)$ is defined by
$$
\hat{\mathbb{E}}[X]:=\widetilde{\mathbb{E}}[\varphi(\sqrt{t_1}\xi_1,\sqrt{t_2-t_1}\xi_2,\cdots,\sqrt{t_n-t_{n-1}}\xi_n)],
$$
for all $X=\varphi(B_{t_1}, B_{t_2}-B_{t_1},\cdots,B_{t_n}-B_{t_{n-1}}), 0\leq t_1<\cdots<t_n<\infty,$
where $\{\xi_i\}_{i=1}^n$ are $d$-dimensional identically distributed random vectors on a sublinear expectation space $(\widetilde{\Omega},\widetilde{\mathcal{H}},\widetilde{\mathbb{E}})$ such that $\xi_i$ is $G$-normal distributed and $\xi_{i+1}$ is independent from $(\xi_1,\cdots,\xi_i)$ for $i=1,\cdots,n-1.$ Then under $\hat{\mathbb{E}}$, the canonical process $B_t=(B_t^1,\cdots,B_t^d)$ is  a $d$-dimensional $G$-Brownian motion in the sense that:
	\begin{itemize}
		\item [{\rm(i)}] $B_0=0$;
		\item [{\rm(ii)}] For each $t,s\geq 0$, the increments $B_{t+s}-B_t$ is independent from $(B_{t_1},\cdots,B_{t_n})$ for each $n\in \mathbb{N}$ and $0\leq t_1\leq \cdots\leq t_n\leq t$;
		\item [{\rm(iii)}]  $B_{t+s}-B_t\overset{d}{=}\sqrt{s}\xi$ for $t,s\geq 0$, where $\xi$ is $G$-normal distributed.
	\end{itemize}

\begin{remark}
	\upshape{
 {\rm(i)}	It is easy to check  that $G$-Brownian motion  is symmetric, i.e., $(-B_t)_{t\geq 0}$ is also a $G$-Brownian motion.

{\rm(ii)} If specially $G(A)=\frac{1}{2}\text{tr}(A)$, then the $G$-expectation is a linear expectation which corresponds to the Wiener measure $P$, i.e., $\hat{\mathbb{E}}=E_{P}$.}
\end{remark}

The conditional $G$-expectation for $X=\varphi(B_{t_1}, B_{t_2}-B_{t_1},\cdots,B_{t_n}-B_{t_{n-1}})$ at $t=t_j$, $1\leq j\leq n$ is defined by
$$
\hat{\mathbb{E}}_{t_j}[X]:=\phi(B_{t_1}, B_{t_2}-B_{t_1},\cdots,B_{t_j}-B_{t_{j-1}}),
$$
where $\phi(x_1, \cdots,x_j)=\hat{\mathbb{E}}[\varphi(x_1, \cdots,x_j, B_{t_{j+1}}-B_{t_j},\cdots,B_{t_n}-B_{t_{n-1}})]$.

For each $p\geq1$, we denote by $L_G^p(\Omega_t)$ ($L_G^p(\Omega)$ resp.) the completion of $L_{ip}(\Omega_t)$ ($L_{ip}(\Omega)$ resp.) under the norm $||X||_p:=(\hat{\mathbb{E}}[|X|^p])^{1/p}$.
The conditional $G$-expectation $\hat{\mathbb{E}}_t[\cdot]$ can be extended continuously to $L_G^1(\Omega)$ and satisfies the following proposition.
\begin{proposition}\label{condition expectation property}
For $X,Y\in L_G^1(\Omega)$, $t,s\geq 0$,
\begin{description}
  \item [{\rm (i)}] $\hat{\mathbb{E}}_t[X]\leq \hat{\mathbb{E}}_t[Y] \ \text{for} \ X\leq Y$;
  \item [{\rm (ii)}] $\hat{\mathbb{E}}_t[\eta]=\eta \ \text{for} \ \eta\in L_G^1(\Omega_t)$;
  \item  [{\rm (iii)}] $\hat{\mathbb{E}}_t[X+Y]\leq \hat{\mathbb{E}}_t[X]+\hat{\mathbb{E}}_t[Y]$;
  \item [{\rm (iv)}] If$\ \eta\in L_G^1(\Omega_t)$ and is bounded, then $ \hat{\mathbb{E}}_t[\eta X]=\eta^+\hat{\mathbb{E}}_t[X]+\eta^-\hat{\mathbb{E}}_t[-X]$;
 \item [{\rm (v)}] $ \hat{\mathbb{E}}_t[\varphi(\eta,X)]=\hat{\mathbb{E}}_t[\varphi(p,X)]_{p=\eta}$, for each $\ \eta\in L_G^1(\Omega_t;\mathbb{R}^d)$, $X\in L_G^1(\Omega;\mathbb{R}^n)$ and  $\varphi\in C_{b.Lip}(\mathbb{R}^{d+n})$;
 \item [{\rm (vi)}] $\hat{\mathbb{E}}_s[\hat{\mathbb{E}}_t[X]]=\hat{\mathbb{E}}_{t\wedge s}[X]$.
\end{description}
\end{proposition}

We define
$$\mathcal{F}_t:=\sigma(B_s:s\leq t) \ \ \ \ \text{and} \ \ \ \ \mathcal{F}:=\bigvee_{t\geq 0}\mathcal{F}_t$$
as well as
$$
L^0(\mathcal{F}_t):=\{X:X\ \text{is} \ \mathcal{F}_t\text{-measurable}\} \ \ \ \ \text{and} \ \ \ \ L^0(\mathcal{F}):=\{X:X\ \text{is} \ \mathcal{F}\text{-measurable}\}.
$$
The following is the representation theorem.
\begin{theorem}(\cite{DHP,HP})\label{DHP representation}
There exists a family  $\mathcal{P}$ of weakly compact  probability measures on $(\Omega,\mathcal{F})$ such that
$$\hat{\mathbb{E}}[X]=\sup_{P\in\mathcal{P}}E_P[X], \qquad \text{for  each}\ X\in L_G^1(\Omega).$$
$\mathcal{P}$ is called a set that represents $\hat{\mathbb{E}}$.
\end{theorem}
\begin{remark}
	\upshape{
Under each $P\in\mathcal{P}$,  the $G$-Brownian motion  $B$ is a martingale.}
\end{remark}
Given $\mathcal{P}$ that represents $\hat{\mathbb{E}}$, we define the capacity
$$c(A):=\sup_{P\in\mathcal{P}} P(A), \ \ \ \ \text{for  each}\ A\in \mathcal{F}.$$
A set $A\in\mathcal{B}(\Omega)$ is said to be \textit{polar} if $c(A)=0$. A property is said to \textit{holds ``quasi-surely''} (\textit{q.s.}) if it holds outside a polar set.
In the following, we do not distinguish two random variables $X$ and $Y$ if $X=Y$ q.s.
\begin{lemma}\label{upward mct for capacity}
Let  $\{A_n\}_{n=1}^\infty$ be a sequence in $\mathcal{B}(\Omega)$ such that $A_n\uparrow A$. Then $c(A_n)\uparrow c(A)$.
\end{lemma}

For each $p\geq1$, we set
$$\mathbb{L}^p(\Omega):=\{X\in {L}^0(\mathcal{F}): \sup_{P\in\mathcal{P}}E_P[|X|^p]<\infty\}$$
and the larger space
$$
\mathcal{L}(\Omega):=\{X\in L^0(\mathcal{F}):E_P[X]\ \text{exists for each }\ P\in\mathcal{P}\}.
$$
We extend the $G$-expectation to $\mathcal{L}(\Omega)$, still denote it by $\hat{\mathbb{E}}$, by setting
$$\hat{\mathbb{E}}[X]:=\sup_{P\in\mathcal{P}}E_P[X],  \ \ \ \ \text{for}\ X\in \mathcal{L}(\Omega).$$
From \cite{DHP}, we know that $\mathbb{L}^p(\Omega)$ is a Banach space under the norm $||\cdot||_p:=(\hat{\mathbb{E}}[|\cdot|^p])^{1/p}$ and $L_G^p(\Omega)\subset \mathbb{L}^p(\Omega)$.
For $\{X_n\}_{n=1}^\infty\subset \mathbb{L}^p(\Omega)$, $X\in \mathbb{L}^p(\Omega)$, we say that $X_n\rightarrow X$ {in} $\mathbb{L}^p$, denoted by $X=\mathbb{L}^p\text{-}\lim_{n\rightarrow\infty}X_n,$ if $\lim_{n\rightarrow\infty}\hat{\mathbb{E}}[|X_n-X|^p]=0$.
\begin{lemma}\label{upward mct for rv}
Let $X_n\in \mathcal{L}(\Omega)$  be a sequence such that $X_n\uparrow X\ q.s.$ and $-\hat{\mathbb{E}}[-X_1]>-\infty$. Then $$\hat{\mathbb{E}}[X_n]\uparrow \hat{\mathbb{E}}[X].$$
\end{lemma}

For each $T>0$ and $p\geq1$, we define
\begin{align*}
\label{}
M_G^{p,0}(0,T):=& \{\eta=\sum_{j=0}^{N-1}\xi_j(\omega)I_{[t_j,t_{j+1})}(t): N\in\mathbb{N},\ 0\leq t_0\leq t_1\leq \cdots\leq t_N\leq T,  \\
&  \ \xi_j\in L_{G}^p(\Omega_{t_j}),\ j=0,1\cdots,N\}.
\end{align*}
For each $\eta\in M_G^{p,0}(0,T)$, set the norm $\|\eta\|_{M_G^{p}}:=(\hat
{\mathbb{E}}[\int_{0}^{T}|\eta_{t}|^{p}dt])^{\frac{1}{p}} $ and denote by
$M_G^{p}(0,T)$ the completion of $M_G^{p,0}(0,T)$ under $\|\cdot
\|_{M_G^{p}} $.

According to \cite{LP,P7}, we can define $\int_0^t\eta_sdB_s^i$, $\int_0^t\xi_sd\langle B^i,B^j\rangle_s$ and $\int_0^t\xi_sds$  for $\eta\in M_G^2(0,T)$ and $\xi\in M_G^1(0,T)$, where $\langle B^i,B^j\rangle$ denotes the cross-variation process, for  $1\leq i,j\leq d$.

\subsection{Stochastic differential equations driven by $G$-Brownian motion}

We consider the following $G$-SDEs: for each given $0\leq t\leq T<\infty$,
\begin{equation}\label{SDE}
\begin{cases}
dX^{t,\xi}_{s}=b(X_{s}^{t,\xi})ds+\sum_{i,j=1}^dh_{ij}(X_{s}^{t,\xi})d\langle B^i,B^j\rangle_s+\sum_{j=1}^{d}\sigma_j(X_{s}^{t,\xi})dB^j_s,\ \ \ \ s\in [t,T],\\
X_{t}^{t,\xi}=\xi,
\end{cases}
\end{equation}
where $\xi\in L_G^p(\Omega_t;\mathbb{R}^n)$, $p\geq 2$ and
$b,h_{ij},\sigma_j:\mathbb{R}^n\rightarrow \mathbb{R}^n$ are given deterministic functions satisfying the following assumptions:
\begin{description}
	\item [(H1)]Symmetry: $h_{ij}=h_{ji}, 1\leq i,j\leq d$;
	\item [(H2)]Lipschitz continuity: there exists a constant $L$ such that for each  $x,x'\in\mathbb{R}^n$,
	$$|b(x)-b(x')|+\sum_{i,j=1}^d|h_{ij}(x)-h_{ij}(x')|+\sum_{j=1}^d |\sigma_j(x)-\sigma_j(x')|\leq L|x-x'|.$$
\end{description}
For simplicity,  $X_{s}^{0,x}$ will be denoted by $X_{s}^{x}$ for $x\in\mathbb{R}^n$. We have the following estimates for $G$-SDE (\ref{SDE}) which can be found in \cite{P7,G1}.
\begin{lemma}\label{GSDE}
	Assume that the conditions $(H1)$ and $(H2)$ hold. Then $G$-SDE (\ref{SDE}) has a unique solution $(X_{s}^{t,\xi})_{s\in[t,T]}\in M^p_G(t,T;\mathbb{R}^n)$.  Moreover, there exists a constant $C$ depending on $p,T,L,G$ such that for any $x,y\in \mathbb{R}^n,\ t,t'\in[0,T]$,
	\begin{equation}\label{SDE sup control}
	\hat{\mathbb{E}}[\sup_{s\in [0,t]}|X^{x}_s|^p]\leq C(1+|x|^p),
	\end{equation}
	\begin{equation}
	\label{SDE3}\hat{\mathbb{E}}[|X_t^x-X^y_{t'}|^p]\leq C(|x-y|^p+(1+|x|^p)|t-t'|^{p/2}).
	\end{equation}
\end{lemma}
Noting that $X^x_s=X^{t,X^x_t}_s$ for $s\geq t$,  we see from Theorem 4.4 in \cite{HJPS1} that
\begin{lemma}\label{HJPS2 lemma}For each given $\varphi\in C_{b.Lip}(\mathbb{R}^{n})$ and $0\leq t\leq T$, we have
	$$\hat{\mathbb{E}}_t[\varphi(X_{t+s}^{x})]=\hat{\mathbb{E}}[\varphi(X_{t+s}^{t,y})]_{y=X_t^x},\ \ \ \ \text{for}\ s\in [0,T-t].$$
\end{lemma}

\section{Construction of the conditional $G$-expectation $\hat{\mathbb{E}}_{\tau+}$}
In this section, we  provide a construction of   the conditional $G$-expectation $\hat{\mathbb{E}}_{\tau+}$ for any optional time $\tau$ and study its properties. This notion is needed in the derivation of strong Markov property for $G$-SDEs in Section 4. We shall also give an application on the reflection principle for $G$-Brownian motion at the end of this section.
\subsection{The construction of  conditional $G$-expectation $\hat{\mathbb{E}}_{\tau+}$ on ${L}_{G}^{1,\tau+}(\Omega)$}
The mapping $\tau:\Omega\rightarrow[0,\infty)$ is called a stopping time if $\{\tau\leq t\}\in \mathcal{F}_t$ for each $t\geq 0$ and an optional time if $\{\tau< t\}\in \mathcal{F}_t$ for each $t\geq 0$. A stopping time is an optional time but the converse may not hold.

For each optional time $\tau$, we define the $\sigma$-field
$$
\mathcal{F}_{\tau+}:=\{A\in\mathcal{F}:A\cap \{\tau<t\}\in\mathcal{F}_t,\ \forall t\geq0\}=\{A\in\mathcal{F}:A\cap \{\tau\leq t\}\in\mathcal{F}_{t+},\ \forall t\geq0\},
$$
where $\mathcal{F}_{t+}=\cap_{s>t}\mathcal{F}_s.$
If $\tau$ is a stopping time, we also define
$$
\mathcal{F}_{\tau}:=\{A\in\mathcal{F}:A\cap \{\tau\leq t\}\in\mathcal{F}_{t},\ \forall t\geq0\}.
$$

Let $\tau$ be an optional time. For each $p\geq1$, we set
$${L}_{G}^{0,p,\tau+}(\Omega)=\{X=\sum_{i=1}^n\xi_iI_{A_i}: \ n\in\mathbb{N},\ \{A_i\}_{i=1}^n\text{ is an}\ \mathcal{F}_{\tau+}\text{-partition of}\ \Omega,\ \xi_i\in L_G^p(\Omega),\ i=1,\cdots,n\}$$
 and denote by ${L}_{G}^{p,\tau+}(\Omega)$  the completion of ${L}_{G}^{0,p,\tau+}(\Omega)$ under the norm $||\cdot||_p$.
In this subsection, we want to define the conditional $G$-expectation  $$
\hat{\mathbb{E}}_{\tau+}:{L}_{G}^{1,\tau+}(\Omega)\rightarrow{L}_{G}^{1,\tau+}(\Omega)\cap L^0(\mathcal{F}_{\tau+}).$$
This will be accomplished
 in three stages by progressively constructing the conditional expectation  on  $L_{ip}(\Omega)$,  $L^1_G(\Omega)$ and finally ${L}_{G}^{1,\tau+}(\Omega)$.

\begin{remark}
\label{ui remark}
	\upshape{	According to Theorem 25 in \cite{DHP}, for $X\in L^1_G(\Omega)$, we have
	\begin{equation}\label{ui property}
	\hat{\mathbb{E}}[|X|I_{\{|X|>N\}}]\rightarrow 0, \ \ \ \ \text{as} \ N\rightarrow \infty.
	\end{equation}
	This, together with a direct calculation, implies that (\ref{ui property}) still holds for $X\in L^{1,\tau+}_G(\Omega)$.}
\end{remark}
In the following, unless stated otherwise, we shall  always assume that the optional time $\tau$ satisfying the following assumption:
\begin{description}
	\item[(H3)] $c(\{\tau>T\})\rightarrow 0$,\ \ \ \  as $ T\rightarrow \infty$.
\end{description}

\subsubsection*{Stage one: $\hat{\mathbb{E}}_{\tau+}$ on $L_{ip}(\Omega)$}
Let $X\in L_{ip}(\Omega)$. The construction of $$\hat{\mathbb{E}}_{\tau+}:L_{ip}(\Omega)\rightarrow L^{1,\tau+}_G(\Omega)\cap L^0(\mathcal{F}_{\tau+})$$ consists of two steps.

\textit{Step 1.}
For any given simple discrete stopping time $\tau$ taking values in $\{t_i:i\geq 1\}$, we define
\begin{equation}\label{Etau for discrete stopping time, lip}
\hat{\mathbb{E}}_{\tau+}[X]:=\sum_{i=1}^{\infty}\hat{\mathbb{E}}_{t_i}[X]I_{\{\tau=t_i\}},
\end{equation}
where a discrete  stopping (or optional) time is \textit{simple} means that $t_i\uparrow\infty$, as $i\rightarrow\infty$. Here we
employ the convention that $t_{n+i}:=t_n+i$, $i\geq 1$, if $\tau$ is a discrete stopping (or optional) time taking finitely many values $\{t_i:i\leq n\}$ with $t_i\leq t_{i+1}$.

\textit{Step 2.} For a general optional time $\tau$, let $\tau_n$ be a sequence of simple discrete stopping times  such that $\tau_n \rightarrow \tau$ uniformly.
We define \begin{equation}\label{890377433993}
\hat{\mathbb{E}}_{\tau+}[X]:=\mathbb{L}^1\text{-}\lim_{n\rightarrow \infty}\hat{\mathbb{E}}_{\tau_n+}[X].
\end{equation}

\begin{proposition}\label{taun Cauchy lemma}
The conditional expectation $\hat{\mathbb{E}}_{\tau+}:L_{ip}(\Omega)\rightarrow L^{1,\tau+}_G(\Omega)\cap L^0(\mathcal{F}_{\tau+})$ is well-defined.
\end{proposition}
In the following, for notation simplicity, we always use $C_X$ to denote the bound of $X$ for any bounded function $X:\Omega\rightarrow\mathbb{R}$. Similarly, for any given bounded, Lipschitz function $\varphi:\mathbb{R}^n\rightarrow\mathbb{R}$, we always use $C_\varphi$ and $L_\varphi$ to denote its bound and Lipschitz constant respectively.

The proof relies on the following lemmas.
We set
$$\Lambda_{\delta,T}:=\{(u_1,u_2):0\leq u_1,u_2\leq T,\ |u_1-u_2|\leq \delta\}.$$ The first three lemmas concern the continuity properties of conditional expectation $\hat{\mathbb{E}}_t$ on $L_{ip}(\Omega)$.

\begin{lemma}\label{Et continuity lemma}
	Let $X=\varphi(B_{t_1},B_{t_2}-B_{t_1},\cdots,B_{t_n}-B_{t_{n-1}})$ for $\varphi\in C_{b.Lip}(\mathbb{R}^{n\times d})$ with $0\leq t_1<t_2<\cdots<t_n<\infty$. Then for any  $T\geq 0$ and $0\leq s_1\leq s_2\leq T$
	, we have
	\begin{equation}\label{8765433667889}
	|\hat{\mathbb{E}}_{s_2}[X]-\hat{\mathbb{E}}_{s_1}[X]|\leq C\{\sup_{(u_1,u_2)\in \Lambda_{s_2-s_1,T}}(|B_{u_2}-B_{u_1}|\wedge 1)+\sqrt{s_2-s_1}\},
	\end{equation}
	where $C$ is a constant depending only on $X$ and $G$.
\end{lemma}
\begin{proof}
	First suppose $s_1,s_2\in [t_i,t_{i+1}]$ for some $0\leq i\leq n$ with the convention that $t_0=0,t_{n+1}=\infty$. By the definition of conditional $G$-expectation on $L_{ip}(\Omega)$, we have
	\begin{equation}\label{78787899}
	\hat{\mathbb{E}}_{s_j}[X]=\psi_j({B_{t_1},\cdots,B_{t_{i}}-B_{t_{i-1}},B_{s_j}-B_{t_{i}}}),\ \ \ \ \text{for}\ j=1,2,
	\end{equation}
	where $$
	\psi_j(x_1,\cdots,x_{i},x_{i+1})=\hat{\mathbb{E}}[\varphi(x_1,\cdots,x_{i},x_{i+1}+B_{t_{i+1}}-B_{s_j},\cdots,B_{t_{n}}-B_{t_{n-1}})].
	$$
    From the sub-additivity of $\hat{\mathbb{E}}$,
	\begin{align*}
	&|\psi_1(x_1,\cdots,x_{i},x_{i+1})-\psi_2(x'_1,\cdots,x'_{i},x'_{i+1})|\\
	&\ \ \leq (L_\varphi(\sum_{j=1}^{i+1}|x_j-x'_j|+\hat{\mathbb{E}}[|B_{s_2}-B_{s_1}|]))\wedge (2C_\varphi)\\
	&\ \ \leq C_1(\sum_{j=1}^{i+1}|x_j-x'_j|\wedge 1+\sqrt{s_2-s_1}),
	\end{align*}
	where $C_1=(L_\varphi(1\vee \hat{\mathbb{E}}[|B_1|]))\vee (2C_\varphi)$.
	Combining this with (\ref{78787899}), we obtain
	\begin{equation}\label{888888}
	|\hat{\mathbb{E}}_{s_2}[X]-\hat{\mathbb{E}}_{s_1}[X]|
	\leq
	C_1(|B_{s_2}-B_{s_1}|\wedge 1+\sqrt{s_2-s_1}).
	\end{equation}
	
	Next, suppose $s_1\in [t_{i},t_{i+1}],s_2\in [t_{j},t_{j+1}]$ for some $j\geq i$.
	Applying estimate (\ref{888888}), we have
	\begin{align*}
	|\hat{\mathbb{E}}_{s_2}[X]-\hat{\mathbb{E}}_{s_1}[X]|
	&\leq |\hat{\mathbb{E}}_{s_2}[X]-\hat{\mathbb{E}}_{t_j}[X]|+|\hat{\mathbb{E}}_{t_j}[X]-\hat{\mathbb{E}}_{t_{j-1}}[X]|+\cdots+|\hat{\mathbb{E}}_{t_{i+1}}[X]-\hat{\mathbb{E}}_{s_1}[X]|\\
	&\leq C_1(|B_{s_2}-B_{t_j}|\wedge 1+\cdots+|B_{t_{i+1}}-B_{s_1}|\wedge 1) + C_1(\sqrt{{s_2}-{t_j}}+\cdots+\sqrt{{t_{i+1}}-{s_1}})\\
	& \leq C\{\sup_{(u_1,u_2)\in \Lambda_{s_2-s_1,T}}(|B_{u_2}-B_{u_1}|\wedge 1)+\sqrt{s_2-s_1}\},
	\end{align*}
	where $C=(n+1)C_1$.
\end{proof}\\
Note that the estimate in the above lemma is universal: the   right-hand side of  estimate (\ref{8765433667889})  depends only on the difference $s_2-s_1$ instead of the values of $s_1$ and $s_2$. Then we can easily get the following discrete stopping time version. A more general form is given in Lemma
\ref{generalized Etau continuity lemma}.
\begin{lemma}\label{Etau continuity lemma}
	Let  $X\in L_{ip}(\Omega)$. Then for any $T,\delta>0$ and  discrete stopping times $\tau,\sigma\leq T$ taking finitely many values such that $|\tau-\sigma|\leq \delta$, we have
	\begin{equation}|\hat{\mathbb{E}}_{\tau+}[X]-\hat{\mathbb{E}}_{\sigma+}[X]|
	\leq C\{\sup_{(u_1,u_2)\in \Lambda_{\delta,T}}(|B_{u_2}-B_{u_1}|\wedge 1)+\sqrt{\delta}\},
	\end{equation}
	where $C$ is a constant depending only on $X$ and $G$.
\end{lemma}
\begin{proof}
Assume $\tau=\sum_{i=1}^nt_iI_{\{\tau=t_i\}},\sigma=\sum_{i=1}^ms_iI_{\{\sigma=s_i\}}$. By the definition (\ref{Etau for discrete stopping time, lip}), we have
\begin{align*}
|\hat{\mathbb{E}}_{\tau+}[X]-\hat{\mathbb{E}}_{\sigma+}[X]|
& =|\sum_{i=1}^{n}\hat{\mathbb{E}}_{t_i}[X]I_{\{\tau=t_i\}}-\sum_{j=1}^{m}\hat{\mathbb{E}}_{s_j}[X]I_{\{\sigma=s_j\}}|\\
&\leq \sum_{i=1}^{n}\sum_{j=1}^{m}|\hat{\mathbb{E}}_{t_i}[X]-\hat{\mathbb{E}}_{s_j}[X]|I_{\{\tau=t_i\}\cap \{\sigma=s_j\} }.
\end{align*}
Then by
Lemma \ref{Et continuity lemma}, there exists  a constant $C$ depending on $X$ and $G$ such that
\begin{align*}
|\hat{\mathbb{E}}_{\tau+}[X]-\hat{\mathbb{E}}_{\sigma+}[X]|
&\leq \sum_{i=1}^{n}\sum_{j=1}^{m}C(\sup_{(u_1,u_2)\in \Lambda_{|t_i-s_j|,T}}(|B_{u_2}-B_{u_1}|\wedge 1)+\sqrt{|t_i-s_j|})I_{\{\tau=t_i\}\cap \{\sigma=s_j\} }\\
&\leq C(\sup_{(u_1,u_2)\in \Lambda_{\delta,T}}(|B_{u_2}-B_{u_1}|\wedge 1)+\sqrt{\delta}).
\end{align*}
The proof is complete.
\end{proof}

\begin{lemma}\label{sup continuity lemma}Let $T>0$ be a given constant. Then
\begin{equation}\label{12345678901}\hat{\mathbb{E}}[\sup_{(u_1,u_2)\in \Lambda_{\delta,T}}(|B_{u_2}-B_{u_1}|\wedge 1)]\downarrow  0,\ \ \ \  \text{as}\ \delta\downarrow 0.
\end{equation}
\end{lemma}
\begin{proof}
Given any $\varepsilon>0$, by the tightness of $\mathcal{P}$, we may pick a compact set $K\subset \Omega_T$ such that $c(K^c)<\varepsilon$. Then by the Arzel\`{a}-Ascoli theorem, there exists a $\delta>0$  such that $|B_{u_1}(\omega)-B_{u_2}(\omega)|\leq \varepsilon$ for $\omega\in K$ and $|u_1-u_2|\leq\delta$, $0\leq u_1,u_2\leq T$.
Consequently,
\begin{align*}
\hat{\mathbb{E}}[\sup_{(u_1,u_2)\in \Lambda_{\delta,T}}(|B_{u_2}-B_{u_1}|\wedge1)]
 \leq \hat{\mathbb{E}}[\sup_{(u_1,u_2)\in \Lambda_{\delta,T}}|B_{u_2}-B_{u_1}|I_K]+c({K^c})
 \leq 2\varepsilon.
\end{align*}
Since $\varepsilon$ can be arbitrarily small, we obtain the lemma.
\end{proof}
\begin{remark}\label{remark after sup continuity lemma}
	\upshape{	From the proof, we  know that the above lemma is still true for a more general case that $\hat{\mathbb{E}}$ is the upper expectation of a tight family of probability measures. To be precise, for any fixed $T$, let $\Omega_T$ be defined as in Section 2, $(B_t)_{0\leq t\leq T}$ be the canonical process and $\hat{\mathbb{E}}=\sup_{P\in\mathcal{P}'}E_P$, where $\mathcal{P}'$ is  a tight family of probability measures  on $\Omega_T$, then (\ref{12345678901}) holds.
	This  generalization will be used in the next section.}
\end{remark}

The following lemma is analogous to the classical one.
 \begin{lemma}\label{pre consistant Etau lemma}
Let $X\in L_{ip}(\Omega)$ and $\tau,\sigma$ be two simple discrete stopping times. Then $\hat{\mathbb{E}}_{(\tau\wedge\sigma)+}[X]=\hat{\mathbb{E}}_{\tau+}[X]$ on $\{\tau\leq \sigma\}$.
 \end{lemma}
 \begin{proof}
Assume $\tau,\sigma$ taking values in $\{t_i:i\geq 1\}$ and $\{s_i:i\geq 1\}$.
Then $$
\hat{\mathbb{E}}_{(\tau\wedge\sigma)+}[X]=\sum_{i,j=1}^\infty \hat{\mathbb{E}}_{t_i\wedge s_j}[X]I_{\{\tau=t_i,\sigma= s_j\}}.
$$
Multiplying $I_{\{\tau\leq \sigma\}}$ on both sides, since $t_i\leq s_j$ on $\{\tau=t_i,\sigma= s_j\}\cap {\{\tau\leq \sigma\}} $, it follows that
$$
I_{\{\tau\leq \sigma\}}\hat{\mathbb{E}}_{(\tau\wedge\sigma)+}[X]=\sum_{i,j=1}^\infty \hat{\mathbb{E}}_{t_i}[X]I_{\{\tau\leq \sigma\}}I_{\{\tau=t_i,\sigma= s_j\}}=\sum_{i=1}^\infty \hat{\mathbb{E}}_{t_i}[X]I_{\{\tau=t_i\}}I_{\{\tau\leq \sigma\}}=I_{\{\tau\leq \sigma\}}\hat{\mathbb{E}}_{\tau+}[X],
$$
which is the desired conclusion.
\end{proof}

\begin{proof}[Proof of Proposition \ref{taun Cauchy lemma}]
Assume $X\in L_{ip}(\Omega)$.
Let $\tau_n$ be a sequence of simple discrete stopping times such that $\tau_n\rightarrow \tau$ uniformly. We need to show that $\hat{\mathbb{E}}_{\tau_n+}[X]$ is a Cauchy sequence in $\mathbb{L}^1$ and the limit is independent of the choice of the approximation sequence $\tau_n$.
Assume $\tau_n=\sum_{i=1}^\infty t^n_iI_{\{\tau_n=t^n_i\}}$ and $|\tau_n-\tau|\leq \delta_n\rightarrow 0$, as $n\rightarrow\infty$. We can take $n_0$ large enough such that $\delta_n\leq 1$ for $n\geq n_0$, and hence $\{\tau\leq T\}\subset \{\tau_n\leq T+1\}$ and $\{\tau\leq T\}\subset \{\tau_m\leq T+1\}$, for $m,n\geq n_0$. Then it follows from  Lemma \ref{pre consistant Etau lemma} that
\begin{equation}\label{23453455676585}
\begin{split}
|\hat{\mathbb{E}}_{\tau_n+}[X]-\hat{\mathbb{E}}_{\tau_{m}+}[X]|
&=|\hat{\mathbb{E}}_{\tau_n+}[X]-\hat{\mathbb{E}}_{\tau_{m}+}[X]|I_{\{\tau\leq T\}}+|\hat{\mathbb{E}}_{\tau_n+}[X]-\hat{\mathbb{E}}_{\tau_{m}+}[X]|I_{\{\tau>T\}}\\
&\leq|\hat{\mathbb{E}}_{(\tau_n\wedge (T+1))+}[X]-\hat{\mathbb{E}}_{(\tau_m\wedge (T+1))+}[X]|I_{\{\tau\leq T\}}+2C_XI_{\{\tau>T\}}.
\end{split}
\end{equation}
 For any $\varepsilon>0$, we  choose $T$ large enough such that $c(\{\tau>T\})\leq \varepsilon$ by (H3). Taking expectation on both sides of (\ref{23453455676585}) and letting $n,m\rightarrow\infty$, we then obtain by Lemma \ref{Etau continuity lemma} and Lemma \ref{sup continuity lemma}
$$
\limsup_{n,m\rightarrow\infty}\hat{\mathbb{E}}[|\hat{\mathbb{E}}_{\tau_n+}[X]-\hat{\mathbb{E}}_{\tau_{m}+}[X]|]\leq 2C_Xc(\{\tau>T\})\leq 2C_X\varepsilon.
$$
Since $\varepsilon$ can be arbitrarily small, this implies $$\lim_{n,m\rightarrow\infty}\hat{\mathbb{E}}[|\hat{\mathbb{E}}_{\tau_n+}[X]-\hat{\mathbb{E}}_{\tau_{m}+}[X]|]= 0.$$
Similar argument shows that if there exists another simple discrete sequences  $\tau'_n$ such that  $\tau'_n\rightarrow \tau$ uniformly, we have
$$\lim_{n\rightarrow\infty}\hat{\mathbb{E}}[|\hat{\mathbb{E}}_{\tau_n+}[X]-\hat{\mathbb{E}}_{\tau'_{n}+}[X]|]=0.$$

Next, for each $n\geq 1$, we set
\begin{equation}\label{approximation tau in KS problem}
\tau_n:=f_n(\tau):=\sum_{i=1}^{\infty}t^n_iI_{\{t^n_{i-1}\leq \tau< t^n_i\}},\ \ \ \ \text{where}\ t^n_i:=\frac{i}{2^n},\ i\geq 0.
\end{equation}
Then we deduce  $\hat{\mathbb{E}}_{\tau_n+}[X]\in  L^{1,\tau+}_G(\Omega)\cap L^0(\mathcal{F}_{\tau+})$ by the observation that
$$
\sum_{i=1}^{m}\hat{\mathbb{E}}_{t^n_i}[X]I_{\{\tau_n=t^n_i\}}\in  L^{0,1,\tau+}_G(\Omega)\cap L^0(\mathcal{F}_{\tau+}),\ \ \ \ \text{for each}\ m\geq 1
$$
and
\begin{align*}
&\hat{\mathbb{E}}[|\sum_{i=1}^{\infty}\hat{\mathbb{E}}_{t^n_i}[X]I_{\{\tau=t^n_i\}}-\sum_{i=1}^{m}\hat{\mathbb{E}}_{t^n_i}[X]I_{\{\tau_n=t^n_i\}}|]\\
&\ \ \leq\hat{\mathbb{E}}[\sum_{i=m+1}^{\infty}|\hat{\mathbb{E}}_{t^n_i}[X]|I_{\{\tau_n=t^n_i\}}]\\
&\ \ \leq C_X\hat{\mathbb{E}}[\sum_{i=m+1}^{\infty}I_{\{\tau_n=t^n_i\}}]\\
&\ \ = C_Xc(\{\tau\geq t^n_m\})\rightarrow 0,  \ \ \ \  \text{as}\ n\rightarrow\infty.
\end{align*}
By the definition (\ref{890377433993}), this  implies $\hat{\mathbb{E}}_{\tau+}[X]\in  L^{1,\tau+}_G(\Omega)\cap L^0(\mathcal{F}_{\tau+})$.

Finally, if $\tau$ is itself a simple discrete stopping time, then $\hat{\mathbb{E}}_{\tau+}$  defined by (\ref{890377433993}) coincides with the one defined by (\ref{Etau for discrete stopping time, lip}) since we can take the approximation sequence $\tau_n\equiv\tau,n\geq 1$ in Step 2.
\end{proof}

Now we give three fundamental properties which are important for the extension of $\hat{\mathbb{E}}_{\tau+}$ to $L^1_G(\Omega)$.
\begin{proposition}\label{lip Etau lemma}
The conditional expectation $\hat{\mathbb{E}}_{\tau+}$ satisfies the following properties: for $X,Y\in L_{ip}(\Omega)$,
\begin{description}
  \item [{\rm (i)}] $\hat{\mathbb{E}}_{\tau+}[X]\leq \hat{\mathbb{E}}_{\tau+}[Y], \ \text{for} \ X\leq Y$;
  \item [{\rm(ii)}] $\hat{\mathbb{E}}_{\tau+}[X+Y]\leq \hat{\mathbb{E}}_{\tau+}[Y]+\hat{\mathbb{E}}_{\tau+}[Y]$;
\item  [{\rm(iii)}] $\hat{\mathbb{E}}[\hat{\mathbb{E}}_{\tau+}[X]]=\hat{\mathbb{E}}[X]$.
\end{description}
\end{proposition}
In order to prove (iii), we need the following proposition. It  is  a generalized version of Proposition 2.5 (vi) in \cite{HJPS}.
\begin{proposition}\label{main proposition}
Let $A_i\in\mathcal{F}_{t_i},i\leq n$ for $0\leq t_1\leq \cdots\leq t_n$ such that $\cup_{i=1}^n A_i=\Omega$ and $A_i\cap A_j=\emptyset$ for $i\neq j$. Then for each $\xi_i\in L_G^{1}(\Omega),\ i\leq n$, we have
\begin{equation}\label{first proposition}
\hat{\mathbb{E}}[\sum_{i=1}^n\xi_iI_{A_i}]=\hat{\mathbb{E}}[\sum_{i=1}^n\hat{\mathbb{E}}_{t_i}[\xi_i]I_{A_i}].
\end{equation}
\end{proposition}
\begin{proof}
	\textit{Step 1.}
	Suppose first $\xi_i\geq 0$, $i=1,\cdots,n$. For any $P\in\mathcal{P}$,  by
	Lemma 17 in \cite{HP1}, we have
	$$E_P[\xi_i|\mathcal{F}_{t_i}]\leq \hat{\mathbb{E}}_{t_i}[\xi_i]\ \ \ \  P\text{-a.s.}$$
	Then
	$$ E_P[\sum_{i=1}^n\xi_iI_{A_i}]
	=E_P[\sum_{i=1}^nE_P[\xi_i|\mathcal{F}_{t_i}]I_{A_i}]\leq E_P[\sum_{i=1}^n\hat{\mathbb{E}}_{t_i}[\xi_i]I_{A_i}]\leq \hat{\mathbb{E}}[\sum_{i=1}^n\hat{\mathbb{E}}_{t_i}[\xi_i]I_{A_i}].$$
	This implies
	$$
	\hat{\mathbb{E}}[\sum_{i=1}^n\xi_iI_{A_i}]=\sup_{P\in\mathcal{P}}E_P[\sum_{i=1}^n\xi_iI_{A_i}]\leq \hat{\mathbb{E}}[\sum_{i=1}^n\hat{\mathbb{E}}_{t_i}[\xi_i]I_{A_i}].$$

	Now we prove the reverse inequality. We only need to show that, for each  $P\in\mathcal{P}$,
	\begin{equation}
	\label{87666879854}E_P[\sum_{i=1}^n\hat{\mathbb{E}}_{t_i}[\xi_i]I_{A_i}]\leq\hat{\mathbb{E}}[\sum_{i=1}^n\xi_iI_{A_i}].
	\end{equation}
	Let $P\in\mathcal{P}$ be given. For $i\leq n$, noting that $A_i,A_i^c \in\mathcal{F}_{t_i}$, we can choose a sequence of  increasing compact sets  $K^{i}_m\subset A_i$, $m\geq 1$ such that $P(A_i\backslash K^{i}_m)\downarrow0$, as $m\uparrow \infty$ and a sequence of  increasing compact sets $\widetilde{K}^i_m\subset A_i^c$, $m\geq 1$ such that $P(A_i^c \backslash\widetilde{K}^i_m)\downarrow0$, as $m\uparrow \infty$.
	Moreover, since $K^{i}_m\cap\widetilde{K}^i_m=\emptyset$ and $K^{i}_m,\widetilde{K}^i_m$ are compact sets, we have
	\begin{equation}
	\label{111111223232}
	\rho_d(K^{i}_m, \widetilde{K}^i_m)>0.
	\end{equation}
	For each $i, m$, by Theorem 1.2 in \cite{Bi} and  (\ref{111111223232}), there exist two sequences $\{\varphi^{i,m}_{l}\}_{l=1}^{\infty},\{\widetilde{\varphi}^{i,m}_{l}\}_{l=1}^\infty\subset C_b(\Omega_{t_i})$ such that $\varphi^{i,m}_{l}\downarrow I_{K^{i}_m}$, $\widetilde{\varphi}^{i,m}_{l}\downarrow I_{\widetilde{K}^i_k}$, as $l\rightarrow\infty$ and
	\begin{equation}\label{3453534}
	\varphi^{i,m}_{l}\cdot\widetilde{\varphi}^{i,m}_{l}=0,\ \ \ \ \text{for all}\ l\geq 1.
	\end{equation}
	Applying the classical monotone convergence theorem under $P$, we have
	\begin{equation}\label{899870776}
	\begin{split}
	E_P[\sum_{i=1}^n\hat{\mathbb{E}}_{t_i}[\xi_i]I_{A_i}]&= E_P[\sum_{i=1}^n\hat{\mathbb{E}}_{t_i}[\xi_i]I_{A_i}\prod_{j=1}^{i-1}I_{A_j^c}]\\
	&=\lim_{m\rightarrow\infty} E_P[\sum_{i=1}^{n}\hat{\mathbb{E}}_{t_i}[\xi_i]I_{K^{i}_m}\prod_{j=1}^{i-1}I_{\widetilde{K}^{j}_m}] \\
	&=\lim_{m\rightarrow\infty}\lim_{l\rightarrow\infty} E_P[\sum_{i=1}^{n}\hat{\mathbb{E}}_{t_i}[\xi_i]{\varphi^{i,m}_{l}}\prod_{j=1}^{i-1}{\widetilde{\varphi}^{j,m}_l}] \\
	&\leq\lim_{m\rightarrow\infty}\lim_{l\rightarrow\infty} \hat{\mathbb{E}}[\sum_{i=1}^{n}\hat{\mathbb{E}}_{t_i}[\xi_i]{\varphi^{i,m}_{l}}\prod_{j=1}^{i-1}{\widetilde{\varphi}^{j,m}_l}] .
	\end{split}
	\end{equation}
	For any fixed $m,l$,  by (vi), (ii), (iv) of Proposition \ref{condition expectation property}, we have
	\begin{equation*}
	\label{2342345334464}
	\begin{split}
	\hat{\mathbb{E}}[\sum_{i=1}^{n}\hat{\mathbb{E}}_{t_i}[\xi_i]{\varphi^{i,m}_{l}}\prod_{j=1}^{i-1}{\widetilde{\varphi}^{j,m}_l}] &= \hat{\mathbb{E}}[\hat{\mathbb{E}}_{t_{n-1}}[\sum_{i=1}^{n}\hat{\mathbb{E}}_{t_i}[\xi_i]{\varphi^{i,m}_{l}}\prod_{j=1}^{i-1}{\widetilde{\varphi}^{j,m}_l}] ]\\
	&=\hat{\mathbb{E}}[\sum_{i=1}^{n-1}\hat{\mathbb{E}}_{t_i}[\xi_i]{\varphi^{i,m}_{l}}\prod_{j=1}^{i-1}{\widetilde{\varphi}^{j,m}_l}+\hat{\mathbb{E}}_{t_{n-1}}[\xi_n{\varphi^{n,m}_{l}}]\prod_{j=1}^{n-1}{\widetilde{\varphi}^{j,m}_l}].
	\end{split}
	\end{equation*}
	By (\ref{3453534}) and  Proposition \ref{condition expectation property} (iv), we note that
	$$
	\hat{\mathbb{E}}_{t_{n-1}}[\xi_{n-1}]{\varphi^{n-1,m}_{l}}+\hat{\mathbb{E}}_{t_{n-1}}[\xi_n{\varphi^{n,m}_{l}}]\widetilde{\varphi}^{n-1,m}_l
	=\hat{\mathbb{E}}_{t_{n-1}}[\xi_{n-1}{\varphi^{n-1,m}_{l}}+\xi_n{\varphi^{n,m}_{l}}\widetilde{\varphi}^{n-1,m}_l].
	$$
	We thus obtain $$\hat{\mathbb{E}}[\sum_{i=1}^{n}\hat{\mathbb{E}}_{t_i}[\xi_i]{\varphi^{i,m}_{l}}\prod_{j=1}^{i-1}{\widetilde{\varphi}^{j,m}_l}]=\hat{\mathbb{E}}[\sum_{i=1}^{n-2}\hat{\mathbb{E}}_{t_i}[\xi_i]{\varphi^{i,m}_{l}}\prod_{j=1}^{i-1}{\widetilde{\varphi}^{j,m}_l}
	+\hat{\mathbb{E}}_{t_{n-1}}[\xi_{n-1}{\varphi^{n-1,m}_{l}}+\xi_n{\varphi^{n,m}_{l}}\widetilde{\varphi}^{n-1,m}_l]\prod_{j=1}^{n-2}{\widetilde{\varphi}^{j,m}_l}].$$
	Repeating this procedure, we conclude that
	\begin{equation}\label{987676655}
	\hat{\mathbb{E}}[\sum_{i=1}^{n}\hat{\mathbb{E}}_{t_i}[\xi_i]{\varphi^{i,m}_{l}}\prod_{j=1}^{i-1}{\widetilde{\varphi}^{j,m}_l}]=\hat{\mathbb{E}}[\hat{\mathbb{E}}_{t_1}[\sum_{i=1}^{n}\xi_i{\varphi^{i,m}_{l}}\prod_{j=1}^{i-1}{\widetilde{\varphi}^{j,m}_l}]]=\hat{\mathbb{E}}[\sum_{i=1}^{n}\xi_i{\varphi^{i,m}_{l}}\prod_{j=1}^{i-1}{\widetilde{\varphi}^{j,m}_l}].
	\end{equation}
	Substituting (\ref{987676655}) into (\ref{899870776}), we arrive at the inequality
	\begin{equation}\label{88776554545}
	E_P[\sum_{i=1}^n\hat{\mathbb{E}}_{t_i}[\xi_i]I_{A_i}] \leq\lim_{m\rightarrow\infty}\lim_{l\rightarrow\infty} \hat{\mathbb{E}}[\sum_{i=1}^{n}\xi_i{\varphi^{i,m}_{l}}\prod_{j=1}^{i-1}{\widetilde{\varphi}^{j,m}_l}]
	.
	\end{equation}
	By Theorem 1.31 in Chap VI of \cite{P7}, we note that
	\begin{align*}
	\label{}
	\lim_{l\rightarrow\infty}\hat{\mathbb{E}}[\sum_{i=1}^{n}\xi_i{\varphi^{i,m}_{l}}\prod_{j=1}^{i-1}{\widetilde{\varphi}^{j,m}_l}]
	&= \hat{\mathbb{E}}[\sum_{i=1}^{n}\xi_iI_{K^{i}_m}\prod_{j=1}^{i-1}I_{\widetilde{K}^{j}_m}]\\
	&\leq\hat{\mathbb{E}}[\sum_{i=1}^{n}\xi_iI_{K^{i}_m}]\\
	&\leq\hat{\mathbb{E}}[\sum_{i=1}^n\xi_iI_{A_i}].
	\end{align*}
	Thus  (\ref{87666879854}) is proved.

	\textit{Step 2.}
	Consider now the general case. We define $\xi^N_i=\xi_i\vee (-N)$
	for constant $N>0$. By Step 1,
	\begin{equation}
	\label{3247839273465}
	\hat{\mathbb{E}}[\sum_{i=1}^n(\xi_i^N+N)I_{A_i}]=\hat{\mathbb{E}}[\sum_{i=1}^n\hat{\mathbb{E}}_{t_i}[\xi^N_i+N]I_{A_i}].
	\end{equation}
	Note that
	$$
	\hat{\mathbb{E}}[\sum_{i=1}^n(\xi^N_i+N)I_{A_i}]=\hat{\mathbb{E}}[\sum_{i=1}^n\xi^N_iI_{A_i}]+N
	$$ and $$\hat{\mathbb{E}}[\sum_{i=1}^n\hat{\mathbb{E}}_{t_i}[\xi^N_i+N]I_{A_i}]=\hat{\mathbb{E}}[\sum_{i=1}^n\hat{\mathbb{E}}_{t_i}[\xi^N_i]I_{A_i}]+N.$$
	Subtracting $N$ from both sides of (\ref{3247839273465}), we obtain
	$$
	\hat{\mathbb{E}}[\sum_{i=1}^n\xi^N_iI_{A_i}]=\hat{\mathbb{E}}[\sum_{i=1}^n\hat{\mathbb{E}}_{t_i}[\xi^N_i]I_{A_i}].
	$$
	Letting $N\rightarrow \infty$ yields (\ref{first proposition})
	by  (\ref{ui property})
\end{proof}

\begin{proof}[Proof of Proposition \ref{lip Etau lemma}]
(i), (ii) follows immediately from  the definition of $\hat{\mathbb{E}}_{\tau+}$ and Proposition \ref{condition expectation property} (i), (iii). We just need to prove (iii).

First suppose $\tau$ is a simple discrete stopping time. By Proposition \ref{main proposition}, noting that $\{\tau=t_i\}\in \mathcal{F}_{t_i},i\geq 1$, we have,
$$
\hat{\mathbb{E}}[\hat{\mathbb{E}}_{\tau+}[X]]=\hat{\mathbb{E}}[\sum_{i=1}^{\infty}\hat{\mathbb{E}}_{t_i}[X]I_{\{\tau=t_i\}}]=\lim_{n\rightarrow\infty}\hat{\mathbb{E}}[\sum_{i=1}^{n}\hat{\mathbb{E}}_{t_i}[X]I_{\{\tau=t_i\}}]=\lim_{n\rightarrow\infty}\hat{\mathbb{E}}[\sum_{i=1}^{n}X I_{\{\tau=t_i\}}]=\hat{\mathbb{E}}[X].
$$

Now we consider the general optional time $\tau$. Taking a simple discrete stopping time sequence $\tau_n\rightarrow\tau$ uniformly, we obtain $$\hat{\mathbb{E}}[\hat{\mathbb{E}}_{\tau+}[X]]=\hat{\mathbb{E}}[\mathbb{L}^1\text{-}\lim_{n\rightarrow\infty}\hat{\mathbb{E}}_{\tau_n+}[X]]=\lim_{n\rightarrow\infty}\hat{\mathbb{E}}[\hat{\mathbb{E}}_{\tau_n+}[X]]= \hat{\mathbb{E}}[X],$$
which is the desired result.
\end{proof}
\subsubsection*{Stage two: $\hat{\mathbb{E}}_{\tau+}$ on $L_G^1(\Omega)$}
We proceed to define $$\hat{\mathbb{E}}_{\tau+}: L_G^1(\Omega)\rightarrow L^{1,\tau+}_G(\Omega)\cap L^0(\mathcal{F}_{\tau+}). $$
Let $X\in L_G^1(\Omega)$. Then there exists a sequence $\{X_n\}_{n=1}^\infty\subset L_{ip}(\Omega)$ such that $X_n\rightarrow X$ in $\mathbb{L}^1$.
We define
$$
\hat{\mathbb{E}}_{\tau+}[X]:=\mathbb{L}^1\text{-}\lim_{n\rightarrow \infty}\hat{\mathbb{E}}_{\tau+}[X_n].
$$
This extension of $\hat{\mathbb{E}}_{\tau+}$ also satisfies the  basic properties in Proposition \ref{lip Etau lemma}.
\begin{proposition}\label{LG1 welldefined lemma}
The conditional expectation $\hat{\mathbb{E}}_{\tau+}:L_G^1(\Omega)\rightarrow L^{1,\tau+}_G(\Omega)\cap L^0(\mathcal{F}_{\tau+})$ is well-defined and satisfies: for $X,Y\in L_{G}^1(\Omega)$,
\begin{description}
  \item [{\rm (i)}] $\hat{\mathbb{E}}_{\tau+}[X]\leq \hat{\mathbb{E}}_{\tau+}[Y], \ \text{for} \ X\leq Y$;
  \item [{\rm (ii)}] $\hat{\mathbb{E}}_{\tau+}[X+Y]\leq \hat{\mathbb{E}}_{\tau+}[Y]+\hat{\mathbb{E}}_{\tau+}[Y]$;
\item  [{\rm (iii)}] $\hat{\mathbb{E}}[\hat{\mathbb{E}}_{\tau+}[X]]=\hat{\mathbb{E}}[X]$.
\end{description}
\end{proposition}
\begin{proof}
(i)-(iii) are obvious by the definition and Proposition \ref{lip Etau lemma}. We just show that $\hat{\mathbb{E}}_{\tau+} $ is well-defined on $L_G^1(\Omega)$.

Let $X\in L_G^1(\Omega)$. Take any $\{X_n\}_{n=1}^\infty\subset L_{ip}(\Omega)$ such that $X_n\rightarrow X$ in $\mathbb{L}^1$. By (i), (ii), (iii) of Proposition \ref{lip Etau lemma}, we have
\begin{align*}
\hat{\mathbb{E}}[|\hat{\mathbb{E}}_{\tau+}[X_n]-\hat{\mathbb{E}}_{\tau+}[X_m]|]
\leq \hat{\mathbb{E}}[\hat{\mathbb{E}}_{\tau+}[|X_n-X_m|]]=\hat{\mathbb{E}}[|X_n-X_m|]\rightarrow 0,\ \ \ \ \text{as}\  n,m\rightarrow\infty.
\end{align*}
Moreover, a similar argument shows that the limit is independent of the choice of the approximation sequence $\{X_n\}_{n=1}^\infty$.
\end{proof}

\subsubsection*{Stage three: $\hat{\mathbb{E}}_{\tau+}$ on $L_G^{1,\tau+}(\Omega)$}
Finally, we define
$$\hat{\mathbb{E}}_{\tau+}: L^{1,\tau+}_G(\Omega)\rightarrow L^{1,\tau+}_G(\Omega)\cap L^0(\mathcal{F}_{\tau+})$$
by two steps.

\textit{Step 1.}
Let $X=\sum_{i=1}^{n}\xi_iI_{A_i}\in L_G^{0,1,\tau+}(\Omega)$, where $\xi_i\in L_G^1(\Omega)$ and  $\{A_i\}_{i=1}^n$ is an $\mathcal{F}_{\tau+}$-partition of $\Omega$.
We define $$
\hat{\mathbb{E}}_{\tau+}[X]:=\sum_{i=1}^{n}\hat{\mathbb{E}}_{\tau+} [\xi_i]I_{A_i}.
$$
Then $\hat{\mathbb{E}}_{\tau+}$ is well-defined by the following lemma.
\begin{lemma}\label{well define lemma}
Let $A\in \mathcal{F}_{\tau+}$ and $\xi,\eta\in L_G^1(\Omega)$. Then $\xi I_{A}\geq \eta I_A$ implies
\begin{equation}\label{34567}
I_{A}\hat{\mathbb{E}}_{\tau+}[\xi]\geq I_{A}\hat{\mathbb{E}}_{\tau+}[\eta].
\end{equation}
\end{lemma}
\begin{proof}
By approximation, we may assume that $\xi,\eta \in L_{ip}(\Omega)$.

We first prove the case that $\tau$ is a simple discrete stopping time taking values in $\{t_i:i\geq 1\}$ and $A\in \mathcal{F}_{\tau}$. Applying Lemma 2.4 in \cite{HJPS}, we have
\begin{equation*}
I_{A}\hat{\mathbb{E}}_{\tau+}[\xi]=\sum_{i=1}^\infty\hat{\mathbb{E}}_{t_i}[\xi]I_{A\cap \{\tau=t_i\}}\geq \sum_{i=1}^\infty\hat{\mathbb{E}}_{t_i}[\eta]I_{A\cap \{\tau=t_i\}}=I_{A}\hat{\mathbb{E}}_{\tau+}[\eta].
\end{equation*}

Now for the general $\tau$, take $\tau_n$ as (\ref{approximation tau in KS problem}). Since $A\in \mathcal{F}_{\tau+}\subset\mathcal{F}_{\tau_n}$, we have
$$
I_{A}\hat{\mathbb{E}}_{\tau+}[\xi]=\mathbb{L}^1\text{-}\lim_{n\rightarrow \infty}I_{A}\hat{\mathbb{E}}_{\tau_n+}[\xi]\geq \mathbb{L}^1\text{-}\lim_{n\rightarrow \infty}I_{A}\hat{\mathbb{E}}_{\tau_n+}[\eta]=I_{A}\hat{\mathbb{E}}_{\tau+}[\eta].
$$
This proves the lemma.
\end{proof}
\begin{proposition}\label{L01tauG lemma}The conditional expectation $\hat{\mathbb{E}}_{\tau+}:L_G^{0,1,\tau+}(\Omega)\rightarrow L^{1,\tau+}_G(\Omega)\cap L^0(\mathcal{F}_{\tau+})$ satisfies: for $X,Y\in L^{0,1,\tau+}_G(\Omega)$,
\begin{description}
  \item [{\rm (i)}] $\hat{\mathbb{E}}_{\tau+}[X]\leq \hat{\mathbb{E}}_{\tau+}[Y], \ \text{for} \ X\leq Y$;
  \item [{\rm (ii)}] $\hat{\mathbb{E}}_{\tau+}[X+Y]\leq \hat{\mathbb{E}}_{\tau+}[Y]+\hat{\mathbb{E}}_{\tau+}[Y]$;
  \item  [{\rm (iii)}] $\hat{\mathbb{E}}[\hat{\mathbb{E}}_{\tau+}[X]]=\hat{\mathbb{E}}[X]$.
\end{description}
\end{proposition}
\begin{proof}
We just prove (iii). The proof for (i), (ii) is trivial.

 First assume that $\tau$ is a simple discrete stopping time taking  values in $\{t_j:j\geq 1\}$ and $X=\sum_{i=1}^{n}\xi_iI_{A_i}$, where $\xi_i\in L_{ip}(\Omega)$ and $\{A_i\}_{i=1}^n$ is an $\mathcal{F}_{\tau}$-partition of $\Omega$. By Proposition \ref{main proposition},
$$
\hat{\mathbb{E}}[\hat{\mathbb{E}}_{\tau+}[X]]=\hat{\mathbb{E}}[\sum_{i=1}^{n}\hat{\mathbb{E}}_{\tau+}[\xi_i]I_{A_i}]=\lim_{m\rightarrow\infty}\hat{\mathbb{E}}[\sum_{i=1}^{n}\sum_{j=1}^{m}\hat{\mathbb{E}}_{t_j}[\xi_i]I_{{A_i}\cap{\{\tau=t_j\}}}]
=\lim_{m\rightarrow\infty}\hat{\mathbb{E}}[\sum_{i=1}^{n}\sum_{j=1}^{m}\xi_iI_{{A_i}\cap{\{\tau=t_j\}}}]=\hat{\mathbb{E}}[X].$$

Next suppose that $\tau$ is an optional time and $X=\sum_{i=1}^{n}\xi_iI_{A_i}$, where $\xi_i\in L_{ip}(\Omega)$ and $\{A_i\}_{i=1}^n$ is an $\mathcal{F}_{\tau+}$-partition of $\Omega$. Taking $\tau_m$ as (\ref{approximation tau in KS problem}), then we derive that
$$
\hat{\mathbb{E}}[\hat{\mathbb{E}}_{\tau+}[X]]=\hat{\mathbb{E}}[\sum_{i=1}^{n}\hat{\mathbb{E}}_{\tau+}[\xi_i]I_{A_i}]
=\lim_{m\rightarrow\infty}\hat{\mathbb{E}}[\sum_{i=1}^{n}\hat{\mathbb{E}}_{\tau_m+}[\xi_i]I_{A_i}]=\lim_{m\rightarrow\infty}\hat{\mathbb{E}}[\hat{\mathbb{E}}_{\tau_m+}[\sum_{i=1}^{n}\xi_iI_{A_i}]]
=\lim_{m\rightarrow\infty}\hat{\mathbb{E}}[X]=\hat{\mathbb{E}}[X].
$$

Consider finally the general case  that $\tau$ is an optional time and $X=\sum_{i=1}^{n}\xi_iI_{A_i}$, where $\xi_i\in L^1_G(\Omega)$ and $\{A_i\}_{i=1}^n$ is an $\mathcal{F}_{\tau+}$-partition of $\Omega$. We can take  sequences $\xi^k_i\in L_{ip}(\Omega)$ such that $\xi^k_i\rightarrow \xi_i$ in $\mathbb{L}^1$, $i\leq n$ to conclude that
$$
\hat{\mathbb{E}}[\hat{\mathbb{E}}_{\tau+}[X]]=\hat{\mathbb{E}}[\sum_{i=1}^{n}\hat{\mathbb{E}}_{\tau+}[\xi_i]I_{A_i}]
=\lim_{k\rightarrow\infty}\hat{\mathbb{E}}[\sum_{i=1}^{n}\hat{\mathbb{E}}_{\tau+}[\xi^k_i]I_{A_i}]
=\lim_{k\rightarrow\infty}\hat{\mathbb{E}}[\sum_{i=1}^{n}\xi^k_iI_{A_i}]=\hat{\mathbb{E}}[X],
$$
as desired.
\end{proof}

\textit{Step 2.}   Let $X\in L_G^{1,\tau+}(\Omega)$. Then there exists a sequence $\{X_n\}_{n=1}^\infty\subset L_G^{0,1,\tau+}(\Omega)$ such that $X_n\rightarrow X$  in $\mathbb{L}^1$. We define
$$
\hat{\mathbb{E}}_{\tau+}[X]:=\mathbb{L}^1\text{-}\lim_{n\rightarrow \infty}\hat{\mathbb{E}}_{\tau+}[X_n].
$$

\begin{proposition}\label{Etau welldefined} The conditional expectation $\hat{\mathbb{E}}_{\tau+}:L_G^{1,\tau+}(\Omega)\rightarrow L^{1,\tau+}_G(\Omega)\cap L^0(\mathcal{F}_{\tau+})$ is well-defined and satisfies the following properties: for $X,Y\in L^{1,\tau+}_G(\Omega)$,
\begin{description}
  \item [{\rm (i)}] $\hat{\mathbb{E}}_{\tau+}[X]\leq \hat{\mathbb{E}}_{\tau+}[Y], \ \text{for} \ X\leq Y$;
  \item [{\rm (ii)}] $\hat{\mathbb{E}}_{\tau+}[X+Y]\leq \hat{\mathbb{E}}_{\tau+}[Y]+\hat{\mathbb{E}}_{\tau+}[Y]$;
  \item  [{\rm (iii)}] $\hat{\mathbb{E}}[\hat{\mathbb{E}}_{\tau+}[X]]=\hat{\mathbb{E}}[X]$.
\end{description}
\end{proposition}
\begin{proof}
It is immediate from the definition of $\hat{\mathbb{E}}_{\tau+}$ on $L_G^{1,\tau+}(\Omega)$ and Proposition \ref{L01tauG lemma}.
\end{proof}
\begin{remark}\label{Etau+ remark}
\upshape{If $G(A)=\frac12{\text{tr}(A)}$, we have  $L^{1}_G(\Omega)=L^{1,\tau+}_G(\Omega)=L^{1}_{P}(\Omega)$ for the Wiener measure $P$, where $L^{1}_{P}(\Omega):=\{X\in\mathcal{F}:\ E_{P}[|X|]<\infty\}$. Moreover, $\hat{\mathbb{E}}_{\tau+}[\cdot]$  is just  the classical conditional expectation ${{E}}_{P}[\cdot|\mathcal{F}_{\tau+}]$.}
\end{remark}
\begin{remark}\label{Etau remark}
\upshape{
	Let $\tau$ be a stopping time satisfying (H3).
	\begin{description}	\item [{\rm (i)}]  We define $L^{1,\tau}_G(\Omega)$ as $L^{1,\tau+}_G(\Omega)$ with $\mathcal{F}_\tau$ in place of $\mathcal{F}_{\tau+}$. By a similar manner,  we can define the conditional expectation at  $\tau$
	$$
	\hat{\mathbb{E}}_{\tau}:{L}_{G}^{1,\tau}(\Omega)\rightarrow{L}_{G}^{1,\tau}(\Omega)\cap L^0(\mathcal{F}_{\tau}),$$
	and analogous properties (throughout this paper)  hold for $\hat{\mathbb{E}}_{\tau}$ and $L^{1,\tau}_G(\Omega)$. For convenience of readers, we sketch  the construction.

		 \textit{Stage one.} Let $X\in L_{ip}(\Omega)$. First for a simple discrete stopping time $\tau$ taking values in $\{t_i:i\geq 1\}$, we define
		\begin{equation*}
		\hat{\mathbb{E}}_{\tau}[X]:=\sum_{i=1}^{\infty}\hat{\mathbb{E}}_{t_i}[X]I_{\{\tau=t_i\}}.
		\end{equation*}
		Then for the general $\tau$, we take  a sequence of simple discrete stopping times $\tau_n$ such that $\tau_n \rightarrow \tau$ uniformly and define \begin{equation*}
		\hat{\mathbb{E}}_{\tau}[X]:=\mathbb{L}^1\text{-}\lim_{n\rightarrow \infty}\hat{\mathbb{E}}_{\tau_n}[X].
		\end{equation*}
		 \textit{Stage two.} Let $X\in L_G^1(\Omega)$. Then there exists a sequence $\{X_n\}_{n=1}^\infty\subset L_{ip}(\Omega)$ such that $X_n\rightarrow X$ in $\mathbb{L}^1$.
		We define
		$$
		\hat{\mathbb{E}}_{\tau}[X]:=\mathbb{L}^1\text{-}\lim_{n\rightarrow \infty}\hat{\mathbb{E}}_{\tau}[X_n].
		$$
		 \textit{Stage three.}  First for $X=\sum_{i=1}^{n}\xi_iI_{A_i}\in L_G^{0,1,\tau}(\Omega)$, where $\xi_i\in L_G^1(\Omega)$ and $\{A_i\}_{i=1}^{n}$ is an $\mathcal{F}_{\tau}$-partition of $\Omega$,
		we define $$
		\hat{\mathbb{E}}_{\tau}[X]:=\sum_{i=1}^{n}\hat{\mathbb{E}}_{\tau} [\xi_i]I_{A_i}.
		$$
		For $X\in L_G^{1,\tau}(\Omega)$, there exists a sequence $\{X_n\}_{n=1}^\infty\subset L_G^{0,1,\tau}(\Omega)$ such that $X_n\rightarrow X$  in $\mathbb{L}^1$. We define
		$$
		\hat{\mathbb{E}}_{\tau}[X]:=\mathbb{L}^1\text{-}\lim_{n\rightarrow \infty}\hat{\mathbb{E}}_{\tau}[X_n].
		$$

    \item [{\rm (ii)}] If $\tau\equiv t$ for some constant $t\geq 0$, then $\hat{\mathbb{E}}_{\tau}$ and $L_G^{1,\tau}(\Omega)$ reduce to $\hat{\mathbb{E}}_{t}$ and $L_G^{1,t}(\Omega)$ defined in \cite{HJPS}.

\item [{\rm (iii)}] In the case that $\tau$ is a stopping time, both $\hat{\mathbb{E}}_{\tau+}$ and $\hat{\mathbb{E}}_{\tau}$ are defined. From the definitions of $\hat{\mathbb{E}}_{\tau+}$ and $\hat{\mathbb{E}}_{\tau}$, it is easy to see that
	$$ \hat{\mathbb{E}}_{\tau+}[X]=\hat{\mathbb{E}}_{\tau}[X],\ \ \  \ \text{for}\ X\in {L}^{1,\tau}_G(\Omega).$$
	
		If $G(A)=\frac12{\text{tr}(A)}$, then  $L^{1}_G(\Omega)=L^{1,\tau}_G(\Omega)=L^{1}_{P}(\Omega)$ and $\hat{\mathbb{E}}_{\tau}[\cdot]$   reduces to the classical conditional expectation  ${{E}}_{P}[\cdot|\mathcal{F}_{\tau}]$, where  $P$ is the Wiener measure.
	\end{description}}
\end{remark}

\subsection{Some further properties of $\hat{\mathbb{E}}_{\tau+}$ on ${L}_{G}^{1,\tau+}(\Omega)$}
Let  $\tau$ be an optional time satisfying (H3). In this subsection, we  describe several interesting properties enjoyed by the conditional expectation $\hat{\mathbb{E}}_{\tau+}$ on $L^{1,{\tau}+}_G(\Omega)$. We begin by observing the following four significant statements.
\begin{proposition}\label{Etau proposition on L1tau}The conditional expectation $\hat{\mathbb{E}}_{\tau+}:L_G^{1,\tau+}(\Omega)\rightarrow L^{1,\tau+}_G(\Omega)\cap L^0(\mathcal{F}_{\tau+})$ satisfies the following properties:
	\begin{description}
		\item [{\rm (i)}]If $X_i\in L^{1,{\tau}+}_G(\Omega)$, $i=1,\cdots,n$ and $\{A_i\}_{i=1}^n$ is an $\mathcal{F}_{\tau+}$-partition of $\Omega$, then $\hat{\mathbb{E}}_{\tau+}[\sum_{i=1}^nX_iI_{A_i}]=\\ \sum_{i=1}^n\hat{\mathbb{E}}_{\tau+}[X_i]I_{A_i}$;
		\item [{\rm (ii)}] If $\tau$ and $\sigma$ are two optional times and $X\in L^{1,{\tau}+}_G(\Omega)$, then $\hat{\mathbb{E}}_{\tau+}[X]I_{\{\tau\leq \sigma\}}=\hat{\mathbb{E}}_{(\tau\wedge\sigma)+}[XI_{\{\tau\leq \sigma\}}]$;
		\item [{\rm (iii)}] If $X\in L^{1,{\tau}+}_G(\Omega)$, then $\hat{\mathbb{E}}_{(\tau\wedge T)+}[XI_{\{\tau\leq T\}}]\rightarrow \hat{\mathbb{E}}_{\tau+}[X]$ in $\mathbb{L}^1$, as $ T\rightarrow\infty$;
		\item [{\rm (iv)}] If $\{\tau_n\}_{n=1}^\infty,\tau$ are optional times such that  $\tau_n\rightarrow\tau$ uniformly, as $n\rightarrow\infty$ and $X\in L^{1,\tau_0+}_G(\Omega) $, where $\tau_0:=\tau \wedge (\wedge_{n=1}^\infty\tau_n)$, then  $\hat{\mathbb{E}}_{\tau_n+}[X]\rightarrow \hat{\mathbb{E}}_{\tau+}[X]$ in $\mathbb{L}^1$, as $n\rightarrow\infty$; in particular,
		if $\tau_n\downarrow\tau$ uniformly, as $n\rightarrow\infty$ and $X\in L^{1,{\tau}+}_G(\Omega)$, then $\hat{\mathbb{E}}_{\tau_n+}[X]\rightarrow \hat{\mathbb{E}}_{\tau+}[X]$ in $\mathbb{L}^1$, as $n\rightarrow\infty$.
	\end{description}
\end{proposition}
\begin{remark}
	\upshape{
For two optional times $\tau$ and $\sigma$, since $A\cap{\{\tau\leq \sigma\}}, A\cap\{\tau=\sigma\} \in \mathcal{F}_{(\tau\wedge \sigma)+}\subset \mathcal{F}_{\sigma+}$ for $A\in \mathcal{F}_{\tau+}$, we have  $XI_{\{\tau\leq \sigma\}}, XI_{\{\tau= \sigma\}}\in L^{1,{(\tau\wedge \sigma)}+}_G(\Omega)\subset L^{1,{\sigma}+}_G(\Omega)$ for $X\in L^{1,{\tau}+}_G(\Omega) $. Hence  the  conditional expectations $\hat{\mathbb{E}}_{(\tau\wedge\sigma)+}[XI_{\{\tau\leq \sigma\}}]$, $ \hat{\mathbb{E}}_{(\tau\wedge\sigma)+}[XI_{\{\tau=\sigma\}}]$, $\hat{\mathbb{E}}_{\sigma+}[XI_{\{\tau\leq \sigma\}}]$ and $ \hat{\mathbb{E}}_{\sigma+}[XI_{\{\tau=\sigma\}}]$  are all   meaningful. }
\end{remark}
The following generalization of Lemma \ref{Etau continuity lemma} is needed for the proof of  Proposition \ref{Etau proposition on L1tau} (iv).
\begin{lemma}\label{generalized Etau continuity lemma}
	Let $X\in L_{ip}(\Omega)$. Then there exists a constant $C$ depending on $X$ and $G$ such that
	$$|\hat{\mathbb{E}}_{\tau+}[X]-\hat{\mathbb{E}}_{\sigma+}[X]|
	\leq C\{\sup_{(u_1,u_2)\in \Lambda_{\delta,T}}(|B_{u_2}-B_{u_1}|\wedge 1)+\sqrt{\delta}\},
	$$
	for any $T,\delta>0$ and optional times $\tau,\sigma\leq T$ such that $|\tau-\sigma|\leq \delta$.
\end{lemma}
\begin{proof}
	Let $\tau_n,\sigma_n\leq T+1$ be two sequences of discrete stopping times taking finitely many values such that $\tau_n\rightarrow\tau,\sigma_n\rightarrow \sigma$ uniformly, as $n\rightarrow\infty$. For any $\varepsilon>0$, we have $|\tau_n-\sigma_n|\leq \delta+\varepsilon$ when $n$ large enough. Then by
	Lemma \ref{Etau continuity lemma}, there exists a constant $C$ depending on $X,G$ such that
	$$|\hat{\mathbb{E}}_{\tau_n+}[X]-\hat{\mathbb{E}}_{\sigma_n+}[X]|
	\leq C\{\sup_{(u_1,u_2)\in \Lambda_{\delta+\varepsilon,T}}(|B_{u_2}-B_{u_1}|\wedge 1)+\sqrt{\delta+\varepsilon}\}.
	$$
	First letting $n\rightarrow \infty$ and then letting $\varepsilon\downarrow 0$, we get the desired conclusion.
\end{proof}

\begin{proof}[Proof of Proposition \ref{Etau proposition on L1tau}]
(i) Let $X_i=\sum_{j=1}^m\xi_j^iI_{B_j^i}\in  L_G^{0,1,\tau+}(\Omega) $, where  $\xi_j^i\in L_G^{1}(\Omega)$ and $\{B_j^i\}_{j=1}^m$ is an $\mathcal{F}_{\tau+}$-partition of $\Omega$.
By the definition of $\hat{\mathbb{E}}_{\tau+}$ on $L^{0,1,\tau+}_G(\Omega)$, we have
	\begin{align*}
	\hat{\mathbb{E}}_{\tau+}[\sum_{i=1}^nX_iI_{A_i}] & =\hat{\mathbb{E}}_{\tau+}[\sum_{i=1}^n\sum_{j=1}^m\xi_j^iI_{B_j^i}I_{A_i}] \\
	&=\hat{\mathbb{E}}_{\tau+}[\sum_{i=1}^n\sum_{j=1}^m\xi_j^iI_{A_i\cap B_j^i}] \\
	& =\sum_{i=1}^n\sum_{j=1}^m\hat{\mathbb{E}}_{\tau+}[\xi_j^i]I_{A_i\cap B_j^i}.
	\end{align*}
Using the definition of $\hat{\mathbb{E}}_{\tau+}$ again, this can be further written as
	$$\sum_{i=1}^n(\sum_{j=1}^m\hat{\mathbb{E}}_{\tau+}[\xi_j^i]I_{B_j^i})I_{A_i}=\sum_{i=1}^n\hat{\mathbb{E}}_{\tau+}[X_i]I_{A_i}.$$
Now the result for the general case of $X_i\in L_G^{1,\tau+}(\Omega) $ follows from a direct limit argument.

(ii)
First assume $X\in L_{ip}(\Omega)$.
Let  $\tau_n:=f_n(\tau),\sigma_n:=f_n(\sigma)$  be  as  in (\ref{approximation tau in KS problem}). Since $\{\tau\leq \sigma\}\subset\{\tau_n\leq \sigma_n\}$, by Lemma \ref{pre consistant Etau lemma}, we have
$$
I_{\{\tau\leq \sigma\}}\hat{\mathbb{E}}_{(\tau_n\wedge\sigma_n)+}[X]=I_{\{\tau\leq \sigma\}}\hat{\mathbb{E}}_{\tau_n+}[X].
$$
Letting $n\rightarrow\infty$, we obtain $$
I_{\{\tau\leq \sigma\}}\hat{\mathbb{E}}_{(\tau\wedge\sigma)+}[X]=I_{\{\tau\leq \sigma\}}\hat{\mathbb{E}}_{\tau+}[X].
$$

Then by a simple approximation, we get for $X\in L^1_G(\Omega)$
$$
I_{\{\tau\leq \sigma\}}\hat{\mathbb{E}}_{(\tau\wedge\sigma)+}[X]=I_{\{\tau\leq \sigma\}}\hat{\mathbb{E}}_{\tau+}[X].
$$
Now it follows from (i) that
$$
\hat{\mathbb{E}}_{(\tau\wedge\sigma)+}[XI_{\{\tau\leq \sigma\}}]=I_{\{\tau\leq \sigma\}}\hat{\mathbb{E}}_{(\tau\wedge\sigma)+}[X]=I_{\{\tau\leq \sigma\}}\hat{\mathbb{E}}_{\tau+}[X].
$$

Next we consider the case $X=\sum_{i=1}^n \xi_iI_{A_i}$, where $\xi_i\in L_G^1(\Omega)$ and $\{A_i\}_{i=1}^n$ is an $\mathcal{F}_{\tau+}$-partition of $\Omega$. We have
\begin{align*}
\hat{\mathbb{E}}_{(\tau\wedge\sigma)+}[XI_{\{\tau\leq \sigma\}}]
&=\hat{\mathbb{E}}_{(\tau\wedge\sigma)+}[\sum_{i=1}^n\xi_iI_{A_i\cap\{\tau\leq \sigma\}}]\\
&=\sum_{i=1}^n\hat{\mathbb{E}}_{(\tau\wedge\sigma)+}[\xi_i]I_{A_i\cap\{\tau\leq \sigma\}}\\
&=\sum_{i=1}^n\hat{\mathbb{E}}_{(\tau\wedge\sigma)+}[\xi_i]I_{\{\tau\leq \sigma\}}I_{A_i}.
\end{align*}
Since $\hat{\mathbb{E}}_{(\tau\wedge\sigma)+}[\xi_i]I_{\{\tau\leq \sigma\}}=\hat{\mathbb{E}}_{\tau+}[\xi_i]I_{\{\tau\leq \sigma\}}$, it follows that
\begin{align*}
\hat{\mathbb{E}}_{(\tau\wedge\sigma)+}[XI_{\{\tau\leq \sigma\}}]
&=\sum_{i=1}^n\hat{\mathbb{E}}_{\tau+}[\xi_i]I_{A_i}I_{\{\tau\leq \sigma\}}\\
&=\hat{\mathbb{E}}_{\tau+}[X]I_{\{\tau\leq \sigma\}}.
\end{align*}

Finally, we obtain the the conclusion for  $X\in L^{1,\tau+}_G(\Omega)$ after an approximation.

(iii)  We first assume that $X$ is bounded. By (i) and (ii),  $$\hat{\mathbb{E}}_{(\tau\wedge T)+}[XI_{\{\tau\leq T\}}]I_{\{\tau\leq T\}}=\hat{\mathbb{E}}_{(\tau\wedge T)+}[XI_{\{\tau\leq T\}}]=\hat{\mathbb{E}}_{\tau+}[X]I_{\{\tau\leq T\}}.$$
Then we directly calculate
\begin{align*}
\hat{\mathbb{E}}[|\hat{\mathbb{E}}_{\tau+}[X]-\hat{\mathbb{E}}_{(\tau\wedge T)+}[XI_{\{\tau\leq T\}}]|]
&=\hat{\mathbb{E}}[|\hat{\mathbb{E}}_{\tau+}[X]-\hat{\mathbb{E}}_{(\tau\wedge T)+}[XI_{\{\tau\leq T\}}]|I_{\{\tau> T\}}]\\
&\leq C_Xc({\{\tau> T\}})\rightarrow 0,\ \ \ \ \text{as} \ T\rightarrow\infty.
\end{align*}

To pass to the case of general $X$ we may argue as follows. Set $X_N:=(X\wedge N)\vee (-N)$ for constant $N>0$. For any $\varepsilon>0$, by Remark \ref{ui remark}, we can take $N$ large enough such that
$$
\hat{\mathbb{E}}[|X-X_N|]\leq\varepsilon.
$$
Then
\begin{align*}
\hat{\mathbb{E}}[|\hat{\mathbb{E}}_{\tau+}[X]-\hat{\mathbb{E}}_{(\tau\wedge T)+}[XI_{\{\tau\leq T\}}]|]
&\leq \hat{\mathbb{E}}[|\hat{\mathbb{E}}_{\tau+}[X]-\hat{\mathbb{E}}_{\tau+}[X_N]|]
+\hat{\mathbb{E}}[|\hat{\mathbb{E}}_{\tau+}[X_N]-\hat{\mathbb{E}}_{(\tau\wedge T)+}[X_NI_{\{\tau\leq T\}}]|]\\
&\ \ \ +\hat{\mathbb{E}}[|\hat{\mathbb{E}}_{(\tau\wedge T)+}[X_NI_{\{\tau\leq T\}}]-\hat{\mathbb{E}}_{(\tau\wedge T)+}[XI_{\{\tau\leq T\}}]|]\\
&\leq 2\varepsilon+\hat{\mathbb{E}}[|\hat{\mathbb{E}}_{\tau+}[X_N]-\hat{\mathbb{E}}_{(\tau\wedge T)+}[X_NI_{\{\tau\leq T\}}]|].
\end{align*}
Letting $T\rightarrow\infty$, we get
$$
\limsup_{T\rightarrow\infty}\hat{\mathbb{E}}[|\hat{\mathbb{E}}_{\tau+}[X]-\hat{\mathbb{E}}_{(\tau\wedge T)+}[XI_{\{\tau\leq T\}}]|]\leq 2\varepsilon,
$$
which implies, since $\varepsilon$ can be arbitrarily small,
 $$
\hat{\mathbb{E}}[|\hat{\mathbb{E}}_{\tau+}[X]-\hat{\mathbb{E}}_{(\tau\wedge T)+}[XI_{\{\tau\leq T\}}]|]\rightarrow 0,\ \ \ \ \text{as}\ T\rightarrow\infty.
$$

(iv) \textit{Step 1.}
Suppose that $\tau_n,\tau\leq T$. We first assume $X\in L^1_G(\Omega)$. For any given $\varepsilon>0$, there exists an $\widetilde{X}\in L_{ip}(\Omega)$ such that
$$\hat{\mathbb{E}}[|\widetilde{X}-X|]\leq \varepsilon.$$ Then
\begin{align*}
\hat{\mathbb{E}}[|\hat{\mathbb{E}}_{\tau_n+}[X]-\hat{\mathbb{E}}_{\tau+}[X]|]
&\leq \hat{\mathbb{E}}[|\hat{\mathbb{E}}_{\tau_n+}[X]-\hat{\mathbb{E}}_{\tau_n+}[\widetilde{X}]|]+\hat{\mathbb{E}}[|\hat{\mathbb{E}}_{\tau_n+}[\widetilde{X}]-\hat{\mathbb{E}}_{\tau+}[\widetilde{X}]|]+\hat{\mathbb{E}}[|\hat{\mathbb{E}}_{\tau+}[\widetilde{X}]-\hat{\mathbb{E}}_{\tau+}[X]|]\\
&\leq 2\varepsilon +\hat{\mathbb{E}}[|\hat{\mathbb{E}}_{\tau_n+}[\widetilde{X}]-\hat{\mathbb{E}}_{\tau+}[\widetilde{X}]|].
\end{align*}
We now let $n\rightarrow \infty$ and use Lemma \ref{generalized Etau continuity lemma} and Lemma \ref{sup continuity lemma} to obtain $$\limsup_{n\rightarrow\infty}\hat{\mathbb{E}}[|\hat{\mathbb{E}}_{\tau_n+}[X]-\hat{\mathbb{E}}_{\tau+}[X]|]\leq 2\varepsilon,$$
which implies
\begin{equation}\label{867655678993} \hat{\mathbb{E}}_{\tau_n+}[X]\rightarrow\hat{\mathbb{E}}_{\tau+}[X] \ \ \ \   \text{in}\  \mathbb{L}^1.
\end{equation}
Next, for $X=\sum_{i=1}^k X_iI_{A_i}$, where $X_i\in L_G^1(\Omega)$ and $\{A_i\}_{i=1}^k$ is an $\mathcal{F}_{\tau_0+}$-partition of $\Omega$, the conclusion follows from (\ref{867655678993}) and the observation that
$$
\hat{\mathbb{E}}[|\hat{\mathbb{E}}_{\tau_n+}[X]-\hat{\mathbb{E}}_{\tau+}[X]|]\leq  \sum_{i=1}^k\hat{\mathbb{E}}[|\hat{\mathbb{E}}_{\tau_n+}[X_i]-\hat{\mathbb{E}}_{\tau+}[X_i]|].
$$
Finally, for $X\in L^{1,\tau_0+}_G(\Omega)$, we can find an  $\widetilde{X}\in L^{0,1,\tau_0+}_G(\Omega)$ such that
$$\hat{\mathbb{E}}[|\widetilde{X}-X|]\leq \varepsilon.$$
Following the argument for the case of $X\in L^1_G(\Omega)$ we can then obtain the conclusion for $L^{1,\tau_0+}_G(\Omega)$.

\textit{Step 2.}
We now  consider the case that $\tau$ is not bounded. Without loss of generality, we can assume $0\leq \tau \vee (\vee_{n=1}^\infty\tau_n)-\tau_0\leq 1$. For any $T>0$,
by (ii),
$$\hat{\mathbb{E}}_{\tau_n+}[X]I_{\{\tau_n\leq T+1\}}=\hat{\mathbb{E}}_{(\tau_n\wedge (T+1))+}[XI_{\{\tau_n\leq T+1\}}].$$
Multiplying $I_{\{\tau_0\leq T \}}$, (i) implies
$$\hat{\mathbb{E}}_{\tau_n+}[X]I_{\{\tau_0\leq T \}}=\hat{\mathbb{E}}_{(\tau_n\wedge (T+1))+}[XI_{\{\tau_0\leq T \}}].$$
Similarly, we have
$$\hat{\mathbb{E}}_{\tau+}[X]I_{\{\tau_0\leq T \}}=\hat{\mathbb{E}}_{(\tau\wedge (T+1))+}[XI_{\{\tau_0\leq T \}}].$$
Let  first $X$ be bounded. We have
\begin{equation}\label{788766654366789}
\begin{split}
|\hat{\mathbb{E}}_{\tau_n+}[X]-\hat{\mathbb{E}}_{\tau+}[X]|
&= |\hat{\mathbb{E}}_{\tau_n+}[X]-\hat{\mathbb{E}}_{\tau+}[X]|I_{\{\tau_0\leq T\}}+|\hat{\mathbb{E}}_{\tau_n+}[X]-\hat{\mathbb{E}}_{\tau+}[X]|I_{\{\tau_0>T\}}\\
&\leq |\hat{\mathbb{E}}_{(\tau_n\wedge (T+1))+}[XI_{\{\tau_0\leq T\}}]-\hat{\mathbb{E}}_{(\tau\wedge (T+1))+}[XI_{\{\tau_0\leq T\}}]|+2C_XI_{\{\tau_0>T\}}.
\end{split}
\end{equation}
For any $\varepsilon>0$, we choose $T$ large enough such that $c({\{\tau_0>T\}})\leq c({\{\tau>T\}})\leq \varepsilon$. Taking expectation $\hat{\mathbb{E}}$ on both sides of (\ref{788766654366789}) and letting $n\rightarrow\infty$, we then obtain
$$\hat{\mathbb{E}}[|\hat{\mathbb{E}}_{\tau_n+}[X]-\hat{\mathbb{E}}_{\tau+}[X]|]\leq
2C_X\varepsilon,$$
which implies the conclusion. If $X$ is not necessarily bounded, we obtain the same conclusion by  a  similar truncation technique as in (iii).
\end{proof}

The next result concerns the pull-out properties.
\begin{proposition}\label{Etau proposition on L1tau2}The conditional expectation $\hat{\mathbb{E}}_{\tau+}$ satisfies:
	\begin{description}
		\item [{\rm (i)}] If $X\in L^{1,{\tau}+}_G(\Omega)$ and $\eta ,Y\in L^{1,{\tau}+}_G(\Omega) \cap L^0(\mathcal{F}_{\tau+})$ such that $\eta$ is bounded,  then $ \hat{\mathbb{E}}_{\tau+}[\eta X +Y]=\eta^+\hat{\mathbb{E}}_{\tau+}[X]+\eta^-\hat{\mathbb{E}}_{\tau+}[-X]+Y$;
		\item [{\rm (ii)}] If $\eta\in L^{1,{\tau+}}_G(\Omega;\mathbb{R}^d)\cap  L^0(\mathcal{F}_{\tau+};\mathbb{R}^d)$, $X\in L^{1,{\tau+}}_G(\Omega;\mathbb{R}^n)$ and  $\varphi\in C_{b.Lip}(\mathbb{R}^{d+n})$, then $ \hat{\mathbb{E}}_{\tau+}[\varphi(\eta,X)]=\hat{\mathbb{E}}_{\tau+}[\varphi(p,X)]_{p=\eta}$.
	\end{description}
\end{proposition}
In the proof of Proposition \ref{Etau proposition on L1tau2}, we shall need the following lemmas. We first study the local property of $\hat{\mathbb{E}}_{\tau+}$.
\begin{lemma}\label{Etau consistency identity0}
	Let  $X\in L^{1,\tau+}_G(\Omega)$ for two optional times $\tau$ and $\sigma$. Then
	\begin{equation} \hat{\mathbb{E}}_{\tau+}[X]I_{\{\tau=\sigma\}}=\hat{\mathbb{E}}_{\sigma+}[XI_{\{\tau=\sigma\}}].
	\end{equation}
\end{lemma}
\begin{proof}
	By Proposition \ref{Etau proposition on L1tau}  (ii),
	\begin{equation*}
	\hat{\mathbb{E}}_{(\tau\wedge\sigma)+}[XI_{\{\tau\leq \sigma\}}]=\hat{\mathbb{E}}_{\tau+}[X]I_{\{\tau\leq \sigma\}}.
	\end{equation*}
	Multiplying $I_{\{\tau=\sigma\}}$ on both sides, we see from  Proposition  \ref{Etau proposition on L1tau} (i)  that
	\begin{equation}
	\label{987739783432}
	\hat{\mathbb{E}}_{(\tau\wedge\sigma)+}[XI_{\{\tau= \sigma\}}]=\hat{\mathbb{E}}_{\tau+}[X]I_{\{\tau=\sigma\}}.
	\end{equation}
	Noting that $XI_{\{\tau= \sigma\}}\in L^{1,\sigma+}_G(\Omega)$, we can apply a similar argument to $\tilde{X}=XI_{\{\tau= \sigma\}},\tilde{\sigma}=\tau,\tilde{\tau}=\sigma$ to obtain
	$$\hat{\mathbb{E}}_{(\tau\wedge\sigma)+}[XI_{\{\tau= \sigma\}}]=\hat{\mathbb{E}}_{\sigma+}[XI_{\{\tau=\sigma\}}].$$
	Combining this with (\ref{987739783432}), we obtain the lemma.
\end{proof}

\begin{lemma}\label{Etau discrete lemma 01}
	Let  $X\in L^{1,\tau+}_G(\Omega)$ for a simple optional time $\tau$ taking values in $\{t_i:i\geq 1 \}$. Then
	$$\hat{\mathbb{E}}_{\tau+}[X]=\sum_{i=1}^\infty\hat{\mathbb{E}}_{t_i+}[XI_{\{\tau=t_i\}}].$$
\end{lemma}
\begin{proof}
	Note that $\{\tau=t_i\}\in \mathcal{F}_{\tau+}$. Applying   Lemma \ref{Etau consistency identity0}, we have
	$$
	\hat{\mathbb{E}}_{\tau+}[X]=\sum_{i=1}^\infty\hat{\mathbb{E}}_{\tau+}[X]I_{\{\tau=t_i\}}=\sum_{i=1}^\infty\hat{\mathbb{E}}_{t_i+}[XI_{\{\tau=t_i\}}].
	$$
\end{proof}

The following deterministic-time version of Proposition \ref{Etau proposition on L1tau2} is also needed.
\begin{lemma}\label{LG1t conditional expectation 2}For each $t\geq 0$, the conditional expectation $\hat{\mathbb{E}}_t$ satisfies the following properties:
	\begin{description}
		\item [{\rm (i)}] If $X\in L_G^{1,t}(\Omega)$ and $\eta,Y\in L_G^{1,t}(\Omega)\cap L^0(\mathcal{F}_t)$ such that $\eta$ is bounded, then $ \hat{\mathbb{E}}_t[\eta X+Y]=\eta^+\hat{\mathbb{E}}_t[X]+\eta^-\hat{\mathbb{E}}_t[-X]+Y$;
		\item [{\rm (ii)}] If $\eta\in L_G^{1,t}(\Omega;\mathbb{R}^{d})\cap L^{0}(\mathcal{F}_t;\mathbb{R}^{d})$, $X\in L_G^{1,t}(\Omega;\mathbb{R}^{n})$, then $ \hat{\mathbb{E}}_t[\varphi(\eta,X)]=\hat{\mathbb{E}}_t[\varphi(p,X)]_{p=\eta}$, for each $\varphi\in C_{b.Lip}(\mathbb{R}^{d+n})$.
	\end{description}
\end{lemma}

\begin{proof}
	We just prove (i). Statement (ii) can be proved similarly.
	
	\textit{Step 1.} We first assume
	$$\eta=\sum_{i=1}^n \eta_i I_{A_i},\ \ Y=\sum_{i=1}^n Y_i I_{A_i},\ \ X=\sum_{i=1}^n X_i I_{A_i},$$
	where $\eta_i, Y_i\in L_G^1(\Omega_t)$, $X_i\in L_G^1(\Omega)$ such that $\eta_i$ is bounded and $\{A_i\}_{i=1}^n$ is an $\mathcal{F}_t$-partition of $\Omega$.
	By the definition of $\hat{\mathbb{E}}_t$ on $L^{0,1,t}_G(\Omega)$ (see Remark \ref{Etau remark}) and properties (ii), (iv) of Proposition \ref{condition expectation property}, we have
	\begin{align*}
	\hat{\mathbb{E}}_t[\eta X+Y] & =\hat{\mathbb{E}}_t[\sum_{i=1}^n(\eta_i X_i+Y_i)I_{A_i}] \\
	& =\sum_{i=1}^n\hat{\mathbb{E}}_t[\eta_i X_i+Y_i]I_{A_i} \\
	&=\sum_{i=1}^n(\eta_i^+\hat{\mathbb{E}}_t[X_i]+\eta_i^-\hat{\mathbb{E}}_t[-X_i]+Y_i)I_{A_i} \\
	&=\eta^+\hat{\mathbb{E}}_t[X]+\eta^-\hat{\mathbb{E}}_t[-X]+Y.
	\end{align*}
	
	\textit{Step 2.} Now we consider the general case. We take a sequence $\{X_n\}_{n=1}^\infty\subset L_G^{0,1,t}(\Omega)$ such that
	$$X_n\rightarrow X \ \ \ \  \text{in}\  \mathbb{L}^1.$$
	Moreover, we define
	\begin{equation}
	\label{980897897487303544}
	\eta_n:=\sum_{-2^n}^{2^n}\frac{kC_\eta}{2^{n}}I_{\{\frac{kC_\eta}{2^{n}}\leq \eta<\frac{(k+1)C_\eta}{2^{n}}\}}\end{equation}
	 and
\begin{equation}
	\label{9805665567303544}
Y_n:=\sum_{-n2^n}^{n2^n-1}\frac{k}{2^{n}}I_{\{\frac{k}{2^{n}}\leq Y<\frac{k+1}{2^{n}}\}}+nI_{\{Y\geq n \}}-nI_{\{Y< -n \}}.
\end{equation}
	Then
	$$|\eta_n-\eta|\leq\frac{C_\eta}{2^n} \ \ \ \ \text{and} \ \ \ \  Y_n\rightarrow Y \ \  \text{in} \  \mathbb{L}^1, \ \text{as}\ n\rightarrow\infty$$
	 since
	 $$\hat{\mathbb{E}}[|Y_n-Y|]\leq \hat{\mathbb{E}}[|Y_n-Y|I_{\{-n\leq Y<n\}}]+
 \hat{\mathbb{E}}[|Y_n-Y|I_{\{|Y|\geq n\}}]\leq \frac{1}{2^{n}}+\hat{\mathbb{E}}[|Y|I_{\{|Y|\geq n\}}]\rightarrow 0, \ \ \ \ \text{as} \ n\rightarrow \infty$$
 because of Remark \ref{ui remark}.
  Applying Step 1, we have
	\begin{equation}\label{345611}
	\hat{\mathbb{E}}_t[\eta_n X_n+Y_n] =\eta_n^+\hat{\mathbb{E}}_t[X_n]+\eta_n^-\hat{\mathbb{E}}_t[-X_n]+Y_n.
	\end{equation}
	We note that
	\begin{align*}
	\hat{\mathbb{E}}[|\eta_nX_n+Y_n-\eta X-Y|]
	&\leq \hat{\mathbb{E}}[|\eta_nX_n-\eta_n X|]+\hat{\mathbb{E}}[|X||\eta_n-\eta |]+\hat{\mathbb{E}}[|Y_n-Y|]\\
	&\leq C_\eta\hat{\mathbb{E}}[|X_n-X|]+\frac{C_\eta}{2^n}\hat{\mathbb{E}}[|X|]+\hat{\mathbb{E}}[|Y_n-Y|]\\
	&\rightarrow 0, \ \ \ \  \text{as} \ n\rightarrow\infty
	\end{align*}
	and similarly,
	$$\hat{\mathbb{E}}[|\eta_n^+\hat{\mathbb{E}}_t[X_n]+\eta_n^
	-\hat{\mathbb{E}}_t[-X_n]+Y_n-(\eta^+\hat{\mathbb{E}}_t[X]+\eta^-\hat{\mathbb{E}}_t[-X]+Y)|]\rightarrow 0,\ \ \ \ \text{as}\ n\rightarrow\infty.$$
    Thus the left-hand side (right-hand side resp.) of (\ref{345611}) converges to the left-hand side (right-hand side resp.) of
    $$ \hat{\mathbb{E}}_t[\eta X+Y]=\eta^+\hat{\mathbb{E}}_t[X]+\eta^-\hat{\mathbb{E}}_t[-X]+Y,$$ which completes the proof.
\end{proof}

\begin{proof}[Proof of Proposition \ref{Etau proposition on L1tau2}]
	We define  $\tau_n$  as (\ref{approximation tau in KS problem}).
	Since $\mathcal{F}_{\tau+}\subset \mathcal{F}_{\tau_n}$, we have $L^{1,{\tau+}}_G(\Omega)\subset L^{1,{\tau_n}}_G(\Omega)$. Thus for any $Z\in L^{1,{\tau+}}_G(\Omega)$, we have $ZI_{\{\tau_n=t^n_i\}}\in L^{1,t^n_i}_G(\Omega)$, and hence $\hat{\mathbb{E}}_{t^n_i+}[ZI_{\{\tau_n=t^n_i\}}]=\hat{\mathbb{E}}_{t^n_i}[ZI_{\{\tau_n=t^n_i\}}]$  according to Remark \ref{Etau remark} (iii).
	Then by Proposition \ref{Etau proposition on L1tau} (iv) and Lemma \ref{Etau discrete lemma 01},
	\begin{align*}
	\hat{\mathbb{E}}_{\tau+}[\eta X +Y]
	&=\mathbb{L}^1\text{-}\lim_{n\rightarrow\infty}\hat{\mathbb{E}}_{\tau_n+}[\eta X +Y]\\
	&=\mathbb{L}^1\text{-}\lim_{n\rightarrow\infty}\sum_{i=1}^{\infty}\hat{\mathbb{E}}_{t^n_i+}[(\eta X +Y)I_{\{\tau_n=t^n_i\}}]\\
	&=\mathbb{L}^1\text{-}\lim_{n\rightarrow\infty}\sum_{i=1}^{\infty}\hat{\mathbb{E}}_{t^n_i}[(\eta X +Y)I_{\{\tau_n=t^n_i\}}]\\
	&=\mathbb{L}^1\text{-}\lim_{n\rightarrow\infty}\sum_{i=1}^{\infty}\hat{\mathbb{E}}_{t^n_i}[\eta I_{\{\tau_n=t^n_i\}} XI_{\{\tau_n=t^n_i\}} +YI_{\{\tau_n=t^n_i\}}].
	\end{align*}
	By Lemma \ref{LG1t conditional expectation 2} (i), we note that
	\begin{align*}
	&\hat{\mathbb{E}}_{t^n_i}[\eta I_{\{\tau_n=t^n_i\}} XI_{\{\tau_n=t^n_i\}} +YI_{\{\tau_n=t^n_i\}}]\\
	&\ \ =\eta^+ I_{\{\tau_n=t^n_i\}}\hat{\mathbb{E}}_{t^n_i}[XI_{\{\tau_n=t^n_i\}}]+\eta^- I_{\{\tau_n=t^n_i\}}\hat{\mathbb{E}}_{t^n_i}[-XI_{\{\tau_n=t^n_i\}}]+YI_{\{\tau_n=t^n_i\}}\\
	&\ \ =\eta^+ \hat{\mathbb{E}}_{t^n_i}[XI_{\{\tau_n=t^n_i\}}]+\eta^-\hat{\mathbb{E}}_{t^n_i}[-XI_{\{\tau_n=t^n_i\}}]+YI_{\{\tau_n=t^n_i\}}.
	\end{align*}
	We thus have
	\begin{align*}
	\hat{\mathbb{E}}_{\tau+}[\eta X +Y]
	&=\mathbb{L}^1\text{-}\lim_{n\rightarrow\infty} (\eta^+\hat{\mathbb{E}}_{\tau_n+}[ X ]+\eta^- \hat{\mathbb{E}}_{\tau_n+}[ -X])+Y\\
	&=\eta^+\hat{\mathbb{E}}_{\tau+}[ X ]+\eta^- \hat{\mathbb{E}}_{\tau+}[ -X]+Y.
	\end{align*}

The property (ii) is proved  similarly.
\end{proof}

\subsection{Extension from below}
For  a sequence $\{X_n\}_{n=1}^\infty$ in $ L^{1,\tau+}_G(\Omega)$ such that $X_n\uparrow X$ q.s., we  can not expect $X\in L^{1,\tau+}_G(\Omega)$  (e.g., $X_n:= n, n\geq 1$). So it is necessary to introduce the extension of $\hat{\mathbb{E}}_{\tau+}$ from below as follows to guarantee the upward monotone convergence.

Let $\tau$ be a given optional time and recall the convention (H3). We set
$$L^{1,\tau+,*}_G(\Omega):=\{X\in L^0(\mathcal{F}):\text{there exists}\ X_n\in L^{1,\tau+}_G(\Omega)\ \text{such that}\ X_n\uparrow X \ q.s.\}.$$
For $X\in L^{1,\tau+,*}_G(\Omega)$, let $\{X_n\}_{n=1}^\infty\subset L^{1,\tau+}_G(\Omega)$ such
that $X_n\uparrow X$  q.s. We define
$$
\hat{\mathbb{E}}_{\tau+}[X]:=\lim_{n\rightarrow\infty}\hat{\mathbb{E}}_{\tau+}[X_n].
$$
\begin{proposition}\label{L1tau* welldefine}
The conditional expectation $\hat{\mathbb{E}}_{\tau+}:L^{1,\tau+,*}_G(\Omega)\rightarrow L^{1,\tau+,*}_G(\Omega)\cap L^0(\mathcal{F}_{\tau+})$ is well-defined and satisfies:
for $X,Y\in L^{1,\tau+,*}_G(\Omega)$,
\begin{description}
	\item [{\rm (i)}] $\hat{\mathbb{E}}_{\tau+}[X]\leq \hat{\mathbb{E}}_{\tau+}[Y], \ \text{for} \ X\leq Y$;
	\item [{\rm (ii)}] $\hat{\mathbb{E}}_{\tau+}[X+Y]\leq \hat{\mathbb{E}}_{\tau+}[Y]+\hat{\mathbb{E}}_{\tau+}[Y]$;
	\item  [{\rm (iii)}] $\hat{\mathbb{E}}[\hat{\mathbb{E}}_{\tau+}[X]]=\hat{\mathbb{E}}[X]$.
\end{description}
\end{proposition}
We need the following lemmas for the proof of the above proposition.
\begin{lemma}\label{L1tau upward mct}
Let $X_n,X\in  L^{1,\tau+}_G(\Omega)$ such that $X_n\uparrow X$ q.s. Then
$
\hat{\mathbb{E}}_{\tau+}[X_n]\uparrow \hat{\mathbb{E}}_{\tau+}[X]\ \text{q.s.}
$
\end{lemma}
\begin{proof}
	Since $X_n\leq X$ implies $\hat{\mathbb{E}}_{\tau+}[{X}_n]\leq \hat{\mathbb{E}}_{\tau+}[{X}]$ by Proposition \ref{Etau welldefined} (i), we have $$\lim_{n\rightarrow\infty}\hat{\mathbb{E}}_{\tau+}[{X}_n] \leq \hat{\mathbb{E}}_{\tau+}[X].$$
	Then it suffices to prove the reverse inequality.
 Assume on the contrary that $\eta:=\lim_{n\rightarrow\infty}\hat{\mathbb{E}}_{\tau+}[X_n]\geq \hat{\mathbb{E}}_{\tau+}[X]$ q.s. does not hold, i.e.,
$$
c(\{\eta<\hat{\mathbb{E}}_{\tau+}[X]\})>0.
$$
Since
$$
D_k:=\{\eta+\frac1k\leq \hat{\mathbb{E}}_{\tau+}[X]\}\cap \{|\eta|\leq k\}\uparrow \{\eta<\hat{\mathbb{E}}_{\tau+}[X]\},
$$
we can take $k$ large enough such that, by Lemma \ref{upward mct for capacity},
$$
c(D_k)>0.
$$
Then by Lemma \ref{upward mct for rv}, Proposition \ref{Etau welldefined} (iii),  Proposition \ref{Etau proposition on L1tau} (i) and Proposition \ref{Etau proposition on L1tau2} (i), we have
\begin{align*}
\hat{\mathbb{E}}[(X+k)I_{D_k}]=\lim_{n\rightarrow\infty}\hat{\mathbb{E}}[(X_n+k)I_{D_k}]=\lim_{n\rightarrow\infty}\hat{\mathbb{E}}[(\hat{\mathbb{E}}_{\tau+}[X_n]+k)I_{D_k}]
=\hat{\mathbb{E}}[(\eta+k) I_{D_k}].
\end{align*}
But
$$
\hat{\mathbb{E}}[(X+k)I_{D_k}]=\hat{\mathbb{E}}[(\hat{\mathbb{E}}_{\tau+}[X]+k)I_{D_k}]\geq  \hat{\mathbb{E}}[(\eta+\frac{1}{k}+k) I_{D_k}],
$$
which is a contradiction by Proposition 29 in \cite{HP1}
\end{proof}

\begin{proof}[Proof of Proposition \ref{L1tau* welldefine}]
   Let $X\in L^{1,\tau+,*}_G(\Omega)$.
   For any $X_n\in L^{1,\tau+}_G(\Omega)$ such that $X_n\uparrow X$ q.s., obviously $\lim_{n\rightarrow\infty}\hat{\mathbb{E}}_{\tau+}[X_n]$ exists. We now show that if moreover there is another sequence $\widetilde{X}_n\in L^{1,{\tau}+}_G(\Omega)$ such that
	$
	\widetilde{X}_n\uparrow X
	$ q.s.,
	it holds
	$$
	\lim_{n\rightarrow\infty}\hat{\mathbb{E}}_{\tau+}[X_n]=\lim_{n\rightarrow\infty}\hat{\mathbb{E}}_{\tau+}[\widetilde{X}_n] \ \ \ \ \text{q.s.}
	$$
Noting that $X_n\wedge \widetilde{X}_m\uparrow X_n$, as $m\rightarrow\infty$, by Lemma \ref{L1tau upward mct}, we have
$$
\hat{\mathbb{E}}_{\tau+}[X_n]=\lim_{m\rightarrow\infty} \hat{\mathbb{E}}_{\tau+}[X_n\wedge \widetilde{X}_m]\leq \lim_{m\rightarrow\infty} \hat{\mathbb{E}}_{\tau+}[\widetilde{X}_m].
$$
This follows
$$
\lim_{n\rightarrow\infty}\hat{\mathbb{E}}_{\tau+}[X_n]\leq \lim_{n\rightarrow\infty}\hat{\mathbb{E}}_{\tau+}[\widetilde{X}_n]
$$
Exchanging $X_n,\widetilde{X}_n$, we get the reverse
$$
\lim_{n\rightarrow\infty}\hat{\mathbb{E}}_{\tau+}[X_n]\geq \lim_{n\rightarrow\infty}\hat{\mathbb{E}}_{\tau+}[\widetilde{X}_n].
$$
Thus
$$
\lim_{n\rightarrow\infty}\hat{\mathbb{E}}_{\tau+}[X_n]= \lim_{n\rightarrow\infty}\hat{\mathbb{E}}_{\tau+}[\widetilde{X}_n].
$$
Therefore, $\hat{\mathbb{E}}_{\tau+}$ is well-defined.

Given the definition of $\hat{\mathbb{E}}_{\tau+}$ on $L^{1,\tau+,*}_G(\Omega)$ and Proposition \ref{Etau welldefined}, the proof for properties (i), (ii), (iii) is straightforward. We shall just omit it.
\end{proof}

\begin{proposition}\label{Etau proposition on L1tau*}The conditional expectation $\hat{\mathbb{E}}_{\tau+}$ on $ L^{1,{\tau}+,*}_G(\Omega)$ satisfies the following properties:
	\begin{description}
		\item [{\rm (i)}]If  $X_i\in L^{1,{\tau}+,*}_G(\Omega)$, $i=1,\cdots,n$ and $\{A_i\}_{i=1}^n$ is an $\mathcal{F}_{\tau+}$-partition of $\Omega$, then $\hat{\mathbb{E}}_{\tau+}[\sum_{i=1}^nX_iI_{A_i}]=\sum_{i=1}^n\hat{\mathbb{E}}_{\tau+}[X_i]I_{A_i}$;
		\item [{\rm (ii)}] If $\tau,\sigma$ are two optional times and $X\in L^{1,\tau+,*}_G(\Omega)$, then $\hat{\mathbb{E}}_{\tau+}[X]I_{\{\tau\leq \sigma\}}=\hat{\mathbb{E}}_{(\tau\wedge\sigma)+}[XI_{\{\tau\leq \sigma\}}]$;
		\item[{\rm (iii)}] If $X\in L^{1,{\tau}+,*}_G(\Omega)$ and $\eta,Y\in L^{1,{\tau}+,*}_G(\Omega)\cap L^0(\mathcal{F}_{\tau+})$ such that $\eta,X$ is nonnegative, then $ \hat{\mathbb{E}}_{\tau+}[\eta X +Y]=\eta\hat{\mathbb{E}}_{\tau+}[X]+Y$;
		\item[{\rm (iv)}]If $X_n\in  L^{1,\tau+,*}_G(\Omega)$ such that $X_n\uparrow X$ q.s., then  $X\in L^{1,\tau+,*}_G(\Omega)$ and
		$
		\hat{\mathbb{E}}_{\tau+}[X_n]\uparrow \hat{\mathbb{E}}_{\tau+}[X]\ \text{q.s.}
		$
	\end{description}
\end{proposition}
\begin{proof}
Statements (i), (ii), (iii) follow directly from   Proposition \ref{Etau proposition on L1tau} (i), (ii), Proposition \ref{Etau proposition on L1tau2} (i) and the definition of $\hat{\mathbb{E}}_{\tau+}$ on $L^{1,\tau+,*}_G(\Omega)$.

(iv)  By Proposition \ref{L1tau* welldefine} (i),  we have
$$
\hat{\mathbb{E}}_{\tau+}[X]\geq \lim_{n\rightarrow\infty}\hat{\mathbb{E}}_{\tau+}[{X}_n].$$ To
 prove the reverse inequality,  for each $X_n$, we take  a sequence $X_n^m\in {L}^{1,\tau+}_G(\Omega)$ such that
$X_n^m\uparrow X_n $, as $m\rightarrow \infty$. We define $\tilde{X}_m:=\vee_{n=1}^mX_n^m\in {L}^{1,\tau+}_G(\Omega)$. Then
$$
\tilde{X}_m\leq \vee_{n=1}^mX_n=X_m\ \ \ \ \text{and}\ \ \ \ \tilde{X}_m\uparrow X,\ \ \text{as}\ m\rightarrow\infty.
$$
It follows from the definition of $\hat{\mathbb{E}}_{\tau+}$ on $L^{1,\tau+,*}_G(\Omega)$ that
$$
\hat{\mathbb{E}}_{\tau+}[X]=\lim_{m\rightarrow\infty}\hat{\mathbb{E}}_{\tau+}[\tilde{X}_m]\leq \lim_{m\rightarrow\infty}\hat{\mathbb{E}}_{\tau+}[{X}_m],
$$
as desired.
\end{proof}
\begin{remark}
	\upshape{
 Let $\tau$ be a stopping time satisfying (H3). We define  $L^{1,\tau,*}_G(\Omega)$ as    ${L^{1,\tau+,*}_G(\Omega)}$ with  $\mathcal{F}_\tau$ replacing $\mathcal{F}_{\tau+}$. We can similarly extend
	$\hat{\mathbb{E}}_{\tau}$ from below to
	$L^{1,\tau,*}_G(\Omega)$ and  similar properties also hold for $\hat{\mathbb{E}}_{\tau}$ on
	$L^{1,\tau,*}_G(\Omega)$. Moreover,
	$$ \hat{\mathbb{E}}_{\tau+}[X]=\hat{\mathbb{E}}_{\tau}[X],\ \ \  \ \text{for}\ X\in {L}^{1,\tau,*}_G(\Omega).$$}
\end{remark}

\subsection{The reflection principle for $G$-Brownian motion}
As an application, we give the following reflection principle for $G$-Brownian motion.
\begin{theorem}
Let $\tau$ be an optional time (without the assumption that $\tau$ satisfies (H3)). Then
$$\widetilde{B}_t:=2B_{t\wedge\tau}-B_t=B_{t\wedge\tau}-(B_t-B_{\tau})I_{\{t>\tau\}},  \ \ \ \ \text{for}\ t\geq0,$$
is still a $G$-Brownian motion.
\end{theorem}
\begin{proof}
It suffices to prove that the two processes have the same finite-dimensional distributions, i.e., for any $0\leq t_1<t_2<\cdots<t_n\leq T<\infty$, we have
\begin{equation}\label{reflect equation}
(\widetilde{B}_{t_1},\widetilde{B}_{t_2}-\widetilde{B}_{t_{1}},\cdots ,\widetilde{B}_{t_n}-\widetilde{B}_{t_{n-1}})\overset{d}{=}(B_{t_1},{B}_{t_2}-{B}_{t_{1}},\cdots ,B_{t_n}-B_{t_{n-1}}).
\end{equation}
Moreover, by replacing $\tau$ with $\tau\wedge T$ we may assume without loss of generality that $\tau\leq T$.

Suppose first that  $\tau$ is a stopping time taking finitely many values. We may assume that $\tau$ also takes values in $\{t_i:i\leq n\}$ since we can refine the partition in (\ref{reflect equation}).
Then by the version of Lemma \ref{Etau discrete lemma 01} for $\hat{\mathbb{E}}_{\tau}$, we have
\begin{align*}
&\hat{\mathbb{E}}_{\tau}[\varphi(\widetilde{B}_{t_1},\widetilde{B}_{t_2}-\widetilde{B}_{t_{1}},\cdots ,\widetilde{B}_{t_n}-\widetilde{B}_{t_{n-1}})]\\
&\ \  =\hat{\mathbb{E}}_{\tau}[\varphi(2B_{t_1\wedge\tau}-B_{t_1},\cdots,2B_{t_n\wedge\tau}-B_{t_n}-(2B_{t_{n-1}\wedge\tau}-B_{t_{n-1}}))]\\
&\ \ =\sum_{i=1}^n\hat{\mathbb{E}}_{t_i}[\varphi(2B_{t_1\wedge\tau}-B_{t_1},\cdots,2B_{t_n\wedge\tau}-B_{t_n}-(2B_{t_{n-1}\wedge\tau}-B_{t_{n-1}}))I_{\{\tau=t_i\}}]\\
&\ \ =\sum_{i=1}^n\hat{\mathbb{E}}_{t_i}[\varphi(2B_{t_1\wedge t_i}-B_{t_1},\cdots,2B_{t_n\wedge t_i}-B_{t_n}-(2B_{t_{n-1}\wedge t_i}-B_{t_{n-1}}))]I_{\{\tau=t_i\}}.
\end{align*}
Note that, for $k\leq i$,
$$2B_{t_k\wedge t_i}-B_{t_k}-(2B_{t_{k-1}\wedge t_i}-B_{t_{k-1}})=B_{t_k}-B_{t_{k-1}},$$
and for $k> i$,
$$
2B_{t_k\wedge t_i}-B_{t_k}-(2B_{t_{k-1}\wedge t_i}-B_{t_{k-1}})=-(B_{t_k}-B_{t_{k-1}})\overset{d}{=}B_{t_k}-B_{t_{k-1}}
$$
because of the symmetry of $G$-Brownian motion.
We see from the definition of conditional expectation $\hat{\mathbb{E}}_{t_i}$ on $L_{ip}(\Omega)$ that
 \begin{align*}
&\hat{\mathbb{E}}_{t_i}[\varphi(2B_{t_1\wedge t_i}-B_{t_1},\cdots,2B_{t_n\wedge t_i}-B_{t_n}-(2B_{t_{n-1}\wedge t_i}-B_{t_{n-1}}))]\\
&\ \ =\hat{\mathbb{E}}_{t_i}[\varphi(B_{t_1},\cdots,B_{t_i}-B_{t_{i-1}},-(B_{t_{i+1}}-B_{t_{i}}),\cdots,-(B_{t_{n}}-B_{t_{n-1}}))]\\
&\ \ =\hat{\mathbb{E}}_{t_i}[\varphi(B_{t_1},\cdots,B_{t_i}-B_{t_{i-1}},B_{t_{i+1}}-B_{t_{i}},\cdots,B_{t_{n}}-B_{t_{n-1}})].
\end{align*}
Therefore,
\begin{align*}
&\hat{\mathbb{E}}_{\tau}[\varphi(\widetilde{B}_{t_1},\widetilde{B}_{t_2}-\widetilde{B}_{t_{1}},\cdots ,\widetilde{B}_{t_n}-\widetilde{B}_{t_{n-1}})]\\
&\ \ =\sum_{i=1}^n\hat{\mathbb{E}}_{t_i}[\varphi(B_{t_1},\cdots,B_{t_i}-B_{t_{i-1}},B_{t_{i+1}}-B_{t_{i}},\cdots,B_{t_{n}}-B_{t_{n-1}})]I_{\{\tau=t_i\}}\\
&\ \ =\hat{\mathbb{E}}_{\tau}[\varphi({B}_{t_1},{B}_{t_2}-{B}_{t_{1}},\cdots ,{B}_{t_n}-{B}_{t_{n-1}})].
\end{align*}
Taking expectation $\hat{\mathbb{E}}$ on both sides, we have
$$
\hat{\mathbb{E}}[\varphi(\widetilde{B}_{t_1},\widetilde{B}_{t_2}-\widetilde{B}_{t_{1}},\cdots ,\widetilde{B}_{t_n}-\widetilde{B}_{t_{n-1}})]
=\hat{\mathbb{E}}[\varphi({B}_{t_1},{B}_{t_2}-{B}_{t_{1}},\cdots ,{B}_{t_n}-{B}_{t_{n-1}})].
$$

Turning to the general optional time $\tau\leq T$, we take a sequence of stopping times $\tau_k\leq T+1$ with finitely many values such that $0\leq \tau_k-\tau\leq \frac1k\downarrow 0$. Then
\begin{equation}\label{798374235739382489}
\hat{\mathbb{E}}[\varphi(2B_{t_1\wedge\tau_k}-B_{t_1},\cdots,2B_{t_n\wedge\tau_k}-B_{t_n}-(2B_{t_{n-1}\wedge\tau_k}-B_{t_{n-1}}))] =\hat{\mathbb{E}}[\varphi(B_{t_1},\cdots,B_{t_n}-B_{t_{n-1}})].
\end{equation}
By a  similar analysis  as in the first paragraph  in the proof of Lemma \ref{Et continuity lemma}, we have for some constant $C$ depending on $\varphi$
\begin{align*}
&\hat{\mathbb{E}}[|\varphi(2B_{t_1\wedge\tau_k}-B_{t_1},\cdots,2B_{t_n\wedge\tau_k}-B_{t_n}-(2B_{t_{n-1}\wedge\tau_k}-B_{t_{n-1}}))\\
&\ \ \ -\varphi(2B_{t_1\wedge\tau}-B_{t_1},\cdots,2B_{t_n\wedge\tau}-B_{t_n}-(2B_{t_{n-1}\wedge\tau}-B_{t_{n-1}}))|]\\
&\ \ \leq C\hat{\mathbb{E}}[\sup_{(u_1,u_2)\in \Lambda_{k^{-1},T+1}}(|B_{u_2}-B_{u_1}|\wedge 1)]\downarrow 0,\ \ \ \ \text{as}\ k\rightarrow \infty.
\end{align*}
Thus (\ref{reflect equation}) follows from letting $k\rightarrow\infty$ in (\ref{798374235739382489}).
\end{proof}

\section{Strong Markov Property for $G$-SDEs}
With the notion of conditional expectation $\hat{\mathbb{E}}_{\tau+}$ in hand, we now turn our attention to the strong Markov property for $G$-SDEs.
We first state the Markov property for $G$-SDEs.
\begin{lemma}\label{markov property of sde}
	For $\varphi\in C_{b.Lip}(\mathbb{R}^{m\times n})$, $0\leq t_1< t_2< \cdots< t_m<\infty$ and  $t\geq 0$, we have $$\hat{\mathbb{E}}_t[\varphi(X_{t+t_1}^{x},X_{t+t_2}^{x},\cdots,X_{t+t_m}^{x})]=\hat{\mathbb{E}}[\varphi(X_{t_1}^{y},X_{t_2}^{y},\cdots,X_{t_m}^{y})]_{y=X_t^x}.$$
\end{lemma}
\begin{proof}
	Since $(B_{t+s}-B_t)_{s\geq 0}$ is still a $G$-Brownian motion and  the coefficients $b,h_{ij},\sigma_j$ in $G$-SDE (\ref{SDE}) are independent of the time variable, we have, for any $s\geq0,\ y\in\mathbb{R}^n$,
	$$
	X_{t+s}^{t,y}\overset{d}{=}X_{s}^{y}.
	$$
	This implies, for $\widetilde{\varphi}\in C_{b.Lip}(\mathbb{R}^{n})$, $$\hat{\mathbb{E}}[\widetilde{\varphi}(X_{t+s}^{t,y})]_{y=X_t^x}=\hat{\mathbb{E}}[\widetilde{\varphi}(X_{s}^{y})]_{y=X_t^x}.$$
	Hence by Lemma \ref{HJPS2 lemma},
	\begin{equation}\label{44444444}
	\hat{\mathbb{E}}_t[\widetilde{\varphi}(X_{t+s}^{x})]=\hat{\mathbb{E}}[\widetilde{\varphi}(X_{s}^{y})]_{y=X_t^x}.
	\end{equation}
	For $\varphi\in C_{b.Lip}(\mathbb{R}^{m\times n})$, by  Proposition \ref{condition expectation property}  (vi) and (v), we have
	\begin{align*}
	\hat{\mathbb{E}}_t[\varphi(X_{t+t_1}^{x},X_{t+t_2}^{x},\cdots,X_{t+t_m}^{x})]
	&=\hat{\mathbb{E}}_t[\hat{\mathbb{E}}_{t+t_{m-1}}[\varphi(X_{t+t_1}^{x},X_{t+t_2}^{x},\cdots,X_{t+t_m}^{x})]]\\
	&=\hat{\mathbb{E}}_t[\overline{\varphi}_{m-1}(X_{t+t_1}^{x},X_{t+t_2}^{x},\cdots,X_{t+t_{m-1}}^{x})],
	\end{align*}
	where $$\overline{\varphi}_{m-1}(y_1,\cdots,y_{m-1}):=\hat{\mathbb{E}}_{t+t_{m-1}}[\varphi(y_1,y_2,\cdots,y_{m-1},X_{t+t_m}^{x})],\ \ \ \  (y_1,\cdots,y_{m-1})\in\mathbb{R}^{(m-1)\times n}.$$
We note  that
	$$\overline{\varphi}_{m-1}(y_1,\cdots,y_{m-1})=\hat{\mathbb{E}}[\varphi(y_1,y_2,\cdots,y_{m-1},X_{t_m-t_{m-1}}^{y'_{m-1}})]_{y'_{m-1}=X_{t+t_{m-1}}^{x}}$$
	by (\ref{44444444}).
Then $$\overline{\varphi}_{m-1}(X_{t+t_1}^{x},X_{t+t_2}^{x},\cdots,X_{t+t_{m-1}}^{x})=\varphi_{m-1}(X_{t+t_1}^{x},X_{t+t_2}^{x},\cdots,X_{t+t_{m-1}}^{x}),$$
	where $$\varphi_{m-1}(y_1,\cdots,y_{m-1}):=\hat{\mathbb{E}}[\varphi(y_1,y_2,\cdots,y_{m-1},X_{t_m-t_{m-1}}^{y_{m-1}})] ,\ \ \ \  (y_1,\cdots,y_{m-1})\in\mathbb{R}^{(m-1)\times n}.$$
Thus we have
$$
\hat{\mathbb{E}}_t[\varphi(X_{t+t_1}^{x},X_{t+t_2}^{x},\cdots,X_{t+t_{m}}^{x})]=\hat{\mathbb{E}}_t[\varphi_{m-1}(X_{t+t_1}^{x},X_{t+t_2}^{x},\cdots,X_{t+t_{m-1}}^{x})].
$$
Repeating this procedure, we get
	\begin{align}
	\hat{\mathbb{E}}_t[\varphi(X_{t+t_1}^{x},X_{t+t_2}^{x},\cdots,X_{t+t_{m}}^{x})]\notag
	&=\hat{\mathbb{E}}_t[\varphi_{m-1}(X_{t+t_1}^{x},X_{t+t_2}^{x},\cdots,X_{t+t_{m-1}}^{x})]\notag\\
	&\ \vdots\label{eq1}\\
	&=\hat{\mathbb{E}}_t[\varphi_{1}(X_{t+t_1}^{x})]\notag\\
	&=\hat{\mathbb{E}}[\varphi_{1}(X_{t_1}^{y})]_{y=X_{t}^x},\notag
	\end{align}
	where
	$$
	\varphi_{m-i}(y_1,\cdots,y_{m-i}):=\hat{\mathbb{E}}[\varphi_{m-(i-1)}(y_1,y_2,\cdots,y_{m-i},X_{t_{m-(i-1)}-t_{m-i}}^{y_{m-i}})],\ \ \ \ 1\leq i\leq m-1.
	$$
	
	Taking $t=0,\ x=y$ in (\ref{eq1}), we obtain
	\begin{equation}\label{eq2}
	\hat{\mathbb{E}}[\varphi(X_{t_1}^{y},X_{t_2}^{y},\cdots,X_{t_m}^{y})]=\hat{\mathbb{E}}[\varphi_{1}(X_{t_1}^{y})],\ \ \ \ \text{for any } y\in\mathbb{R}^n.
	\end{equation}
	This, combining with (\ref{eq1}), proves the corollary.
\end{proof}

We now give the strong Markov property for $G$-SDEs. It  generalizes the well-known strong Markov property  for classical SDEs to $G$-SDEs in the framework of nonlinear $G$-expectation. We set $\Omega':=C([0,\infty);\mathbb{R}^n)$ with the distance $\rho_n$ and denote by $B'$  the corresponding canonical process.   Recall that we always assume that the optional time $\tau$ satisfies (H3).
\begin{theorem}\label{main theorem}
	Let $(X^x_t)_{t\geq 0}$ be the solution of $G$-SDE (\ref{SDE}) satisfying (H1), (H2) and  $\tau$ be an optional time. Then for each $\varphi\in C_{b.Lip}(\mathbb{R}^{m\times n})$ and $0\leq t_1\leq \cdots\leq t_m=:T'<\infty$, we have
	\begin{equation}\label{strong markov for sde}
	\hat{\mathbb{E}}_{\tau+}[\varphi(X_{\tau+t_1}^x,\cdots,X_{\tau+t_m}^x)]
	=\hat{\mathbb{E}}[\varphi(X_{t_1}^y,\cdots,X_{t_m}^y)]_{y=X_\tau^x}.
	\end{equation}
\end{theorem}

We first need the following lemma to justify that the conditional expectation on the left-hand side of (\ref{strong markov for sde}) is meaningful. We denote the paths for a process $Y$ by
$
Y_\cdot:=(Y_t)_{t\geq 0}.
$
\begin{lemma}\label{belong to L1tau lemma 1}
We have
	\begin{equation}
	\varphi(X_{\tau+t_1}^x,\cdots,X_{\tau+t_m}^x)\in L^{1,\tau+}_G(\Omega).
	\end{equation}
\end{lemma}
\begin{proof}
	\textit{Step 1.} First assume $\tau\leq T$. Take discrete stopping time  $\tau_n\leq T+1$ as (\ref{approximation tau in KS problem}).
	By the definition of $L^{0,1,\tau+}_G(\Omega)$, we have $$\varphi(X_{\tau_n+t_1}^x,\cdots,X_{\tau_n+t_m}^x)\in L^{0,1,\tau+}_G(\Omega).$$
	Then it suffices to show that
	\begin{equation}\label{9876856758980978675}
	\hat{\mathbb{E}}[|\varphi(X_{\tau_n+t_1}^x,\cdots,X_{\tau_n+t_m}^x)-\varphi(X_{\tau+t_1}^x,\cdots,X_{\tau+t_m}^x)|]\rightarrow 0, \ \ \ \  \text{as}\ n\rightarrow\infty.
	\end{equation}
	
Consider now the mapping $\Omega\overset{X^x_\cdot}{\longrightarrow}\Omega'$. By   Lemma \ref{GSDE} (\ref{SDE3}), for each $T_1\geq 0$, there exists a constant $C_{T_1}$ (depending on $T_1$)  such that for each $t,s\leq T_1 $,
	$$E_P[|X^x_t-X^x_s|^4]\leq\hat{\mathbb{E}}[|X_t^x-X_s^x|^4]\leq C_{T_1}|t-s|^2,\ \ \ \ \text{for each} \ P\in\mathcal{P}.$$
	Then  we can apply  the well-known Kolmogorov's moment criterion for tightness (see, e.g.,  Problem 2.4.11 in \cite{KS}) to conclude  that the induced probability family $\{P\circ (X_{\cdot}^x)^{-1}: P\in\mathcal{P}\}$ is tight on $\Omega'$.
	We denote the induced capacity  by $c^x_2:=\sup_{P\in\mathcal{P}} P\circ (X_{\cdot}^x)^{-1} $ and the induced  sublinear expectation by $\hat{\mathbb{E}}^x_2:=\sup_{P\in\mathcal{P}}E_{P\circ (X_{\cdot}^x)^{-1}}$. Then
	\begin{align*}
	&\hat{\mathbb{E}}[|\varphi(X_{\tau_n+t_1}^x,\cdots,X_{\tau_n+t_m}^x)-\varphi(X_{\tau+t_1}^x,\cdots,X_{\tau+t_m}^x)|]\\
	&\ \ \leq \hat{\mathbb{E}}[\sup_{ s,s'\in \Lambda_{2^{-n},T+1}}|\varphi(X_{s'+t_1}^x,\cdots,X_{s'+t_m}^x)-\varphi(X_{s+t_1}^x,\cdots,X_{s+t_m}^x)|]\\
	&\ \ =\hat{\mathbb{E}}^x_2[\sup_{ s,s'\in \Lambda_{{2^{-n},T+1}}}|\varphi(B'_{s'+t_1},\cdots,B'_{s'+t_m})-\varphi(B'_{s+t_1},\cdots,B'_{s+t_m})|].
	\end{align*}
	Proceeding similarly  to  the first paragraph in proof of Lemma \ref{Et continuity lemma}, we   obtain for some constant $C$ depending on $\varphi$
	$$
	\hat{\mathbb{E}}^x_2[\sup_{ s,s'\in \Lambda_{{{2^{-n},T+1}}}}|\varphi(B'_{s'+t_1},\cdots,B'_{s'+t_m})-\varphi(B'_{s+t_1},\cdots,B'_{s+t_m})|]\leq C\hat{\mathbb{E}}^x_2[\sup_{ s,s'\in \Lambda_{{2^{-n},T+1+T'}}}(|B'_s-B'_{s'}|\wedge 1)],
	$$
which converges to $0$ as $n\rightarrow \infty$ by  Remark \ref{remark after sup continuity lemma}.

	\textit{Step 2.} For the general case, by Step 1, we have
	$$
	\varphi(X_{\tau\wedge T+t_1}^x,\cdots,X_{\tau\wedge T+t_m}^x)\in L^{1,(\tau\wedge T)+}_G(\Omega)\subset L^{1,\tau+}_G(\Omega)
	.
	$$
	Note that
	$$
	\hat{\mathbb{E}}[|\varphi(X_{\tau\wedge T+t_1}^x,\cdots,X_{\tau\wedge T+t_m}^x)-\varphi(X_{\tau+t_1}^x,\cdots,X_{\tau+t_m}^x)|]\leq 2C_\varphi c(\{\tau>T\})\rightarrow 0, \ \ \ \  \text{as}\ T\rightarrow\infty.
	$$
	The result now follows.
\end{proof}

\begin{proof}[Proof of Theorem \ref{main theorem}]
	Let $\tau\leq T$. We define $\tau_n$ as (\ref{approximation tau in KS problem}). Then $\tau_n\leq T+1$ takes finitely values $\{t^n_i:i\leq d_n\}$ with  $d_n:=[2^nT]+1$.
	By (\ref{9876856758980978675}) and Proposition \ref{Etau proposition on L1tau} (iv), we have
	\begin{equation*}
	\begin{split}
	&\hat{\mathbb{E}}[|\hat{\mathbb{E}}_{\tau_n+}[\varphi(X_{\tau_n+t_1}^x,\cdots,X_{\tau_n+t_m}^x)]-\hat{\mathbb{E}}_{\tau+}[\varphi(X_{\tau+t_1}^x,\cdots,X_{\tau+t_m}^x)]|]\\
	&\ \ \leq \hat{\mathbb{E}}[|\hat{\mathbb{E}}_{\tau_n+}[\varphi(X_{\tau_n+t_1}^x,\cdots,X_{\tau_n+t_m}^x)]-\hat{\mathbb{E}}_{\tau_n+}[\varphi(X_{\tau+t_1}^x,\cdots,X_{\tau+t_m}^x)]|]\\
	&\ \ \ \ \ +\hat{\mathbb{E}}[|\hat{\mathbb{E}}_{\tau_n+}[\varphi(X_{\tau+t_1}^x,\cdots,X_{\tau+t_m}^x)]-\hat{\mathbb{E}}_{\tau+}[\varphi(X_{\tau+t_1}^x,\cdots,X_{\tau+t_m}^x)]|]\\
	&\ \ \leq \hat{\mathbb{E}}[|\varphi(X_{\tau_n+t_1}^x,\cdots,X_{\tau_n+t_m}^x)-\varphi(X_{\tau+t_1}^x,\cdots,X_{\tau+t_m}^x)|]\\
	&\ \ \ \ \ +\hat{\mathbb{E}}[|\hat{\mathbb{E}}_{\tau_n+}[\varphi(X_{\tau+t_1}^x,\cdots,X_{\tau+t_m}^x)]-\hat{\mathbb{E}}_{\tau+}[\varphi(X_{\tau+t_1}^x,\cdots,X_{\tau+t_m}^x)]|]\\
	&\ \ \rightarrow 0,  \ \ \ \ \text{as}\ n\rightarrow \infty.
	\end{split}
	\end{equation*}
	Moreover, since  $\varphi(X_{\tau_n+t_1}^x,\cdots,X_{\tau_n+t_m}^x)\in  L^{1,\tau_n}_G(\Omega)$, by  Remark \ref{Etau remark}, we have $$\hat{\mathbb{E}}_{\tau_n+}[\varphi(X_{\tau_n+t_1}^x,\cdots,X_{\tau_n+t_m}^x)]=\hat{\mathbb{E}}_{\tau_n}[\varphi(X_{\tau_n+t_1}^x,\cdots,X_{\tau_n+t_m}^x)].$$
	Combining these with the version of Lemma \ref{Etau discrete lemma 01} for $\hat{\mathbb{E}}_{\tau_n}$, we have
	\begin{align*}
	\hat{\mathbb{E}}_{\tau+}[\varphi(X_{\tau+t_1}^x,\cdots,X_{\tau+t_m}^x)]
	&=\mathbb{L}^1\text{-}\lim_{n\rightarrow \infty}\hat{\mathbb{E}}_{\tau_n}[\varphi(X_{\tau_n+t_1}^x,\cdots,X_{\tau_n+t_m}^x)]\\
	&=\mathbb{L}^1\text{-}\lim_{n\rightarrow \infty}\sum_{i=1}^{d_n}\hat{\mathbb{E}}_{t^n_i}[\varphi(X_{{t^n_i}+t_1}^x,\cdots,X_{{t^n_i}+t_m}^x)]I_{\{\tau_n={t^n_i}\}}.
	\end{align*}
	Note that from Lemma \ref{markov property of sde}
	$$
	\hat{\mathbb{E}}_{t^n_i}[\varphi(X_{{t^n_i}+t_1}^x,\cdots,X_{{t^n_i}+t_m}^x)]=\hat{\mathbb{E}}[\varphi(X_{t_1}^y,\cdots,X_{t_m}^y)]_{y=X^x_{t^n_i}},
	$$
	We thus obtain
	\begin{align*}
	\hat{\mathbb{E}}_{\tau+}[\varphi(X_{\tau+t_1}^x,\cdots,X_{\tau+t_m}^x)]
	&=\mathbb{L}^1\text{-}\lim_{n\rightarrow \infty}\sum_{i=1}^{d_n}\hat{\mathbb{E}}[\varphi(X_{t_1}^y,\cdots,X_{t_m}^y)]_{y=X^x_{t^n_i}}I_{\{\tau_n={t^n_i}\}}\\
	&=\mathbb{L}^1\text{-}\lim_{n\rightarrow \infty}\hat{\mathbb{E}}[\varphi(X_{t_1}^y,\cdots,X_{t_m}^y)]_{y=X^x_{\tau_n}}\\
	&=\hat{\mathbb{E}}[\varphi(X_{t_1}^y,\cdots,X_{t_m}^y)]_{y=X^x_{{\tau}}},
	\end{align*}
	where the last equality is derived from a proof similar to that  of Lemma \ref{belong to L1tau lemma 1} by using (\ref{SDE3}) of Lemma \ref{GSDE}    for spatial variables.
	
Now for the general $\tau$, applying Step 1, we have
	\begin{equation}
	\label{23454464655}
	\hat{\mathbb{E}}_{(\tau\wedge T)+}[\varphi(X_{\tau\wedge T+t_1}^x,\cdots,X_{\tau\wedge T+t_m}^x)]
	=\hat{\mathbb{E}}[\varphi(X_{t_1}^y,\cdots,X_{t_m}^y)]_{y=X^x_{\tau\wedge T}}.
	\end{equation}
	Since $\varphi(X_{\tau+t_1}^x,\cdots,X_{\tau+t_m}^x)\in L^{1,\tau+}_G(\Omega)$ by Lemma \ref{belong to L1tau lemma 1}, we can apply  Proposition \ref{Etau proposition on L1tau} (iii)  to obtain
	\begin{align*}
	&\hat{\mathbb{E}}[|\hat{\mathbb{E}}_{(\tau\wedge T)+}[\varphi(X_{\tau\wedge T+t_1}^x,\cdots,X_{\tau\wedge T+t_m}^x)]-\hat{\mathbb{E}}_{\tau+}[\varphi(X_{\tau+t_1}^x,\cdots,X_{\tau+t_m}^x)]|]\\
	&\ \ \leq \hat{\mathbb{E}}[|\hat{\mathbb{E}}_{(\tau\wedge T)+}[\varphi(X_{\tau\wedge T+t_1}^x,\cdots,X_{\tau\wedge T+t_m}^x)]-\hat{\mathbb{E}}_{(\tau\wedge T)+}[\varphi(X_{\tau+t_1}^x,\cdots,X_{\tau+t_m}^x)I_{\{\tau\leq T\}}]|]\\
	&\ \  \ \ \ +\hat{\mathbb{E}}[|\hat{\mathbb{E}}_{(\tau\wedge T)+}[\varphi(X_{\tau+t_1}^x,\cdots,X_{\tau+t_m}^x)I_{\{\tau\leq T\}}]-\hat{\mathbb{E}}_{\tau+}[\varphi(X_{\tau+t_1}^x,\cdots,X_{\tau+t_m}^x)]|]\\
	&\ \ \leq \hat{\mathbb{E}}[|\hat{\mathbb{E}}_{(\tau\wedge T)+}[\varphi(X_{\tau\wedge T+t_1}^x,\cdots,X_{\tau\wedge T+t_m}^x)]-\hat{\mathbb{E}}_{(\tau\wedge T)+}[\varphi(X_{\tau+t_1}^x,\cdots,X_{\tau+t_m}^x)I_{\{\tau\leq T\}}]|I_{\{\tau\leq T\}}]\\
	&\ \  \ \ \ +\hat{\mathbb{E}}[|\hat{\mathbb{E}}_{(\tau\wedge T)+}[\varphi(X_{\tau\wedge T+t_1}^x,\cdots,X_{\tau\wedge T+t_m}^x)]-\hat{\mathbb{E}}_{(\tau\wedge T)+}[\varphi(X_{\tau+t_1}^x,\cdots,X_{\tau+t_m}^x)I_{\{\tau\leq T\}}]|I_{\{\tau> T\}}]\\
	&\ \ \ \ \ +\hat{\mathbb{E}}[|\hat{\mathbb{E}}_{(\tau\wedge T)+}[\varphi(X_{\tau+t_1}^x,\cdots,X_{\tau+t_m}^x)I_{\{\tau\leq T\}}]-\hat{\mathbb{E}}_{\tau+}[\varphi(X_{\tau+t_1}^x,\cdots,X_{\tau+t_m}^x)]|]\\
	&\ \ \leq C_\varphi c({\{\tau>T\}})+\hat{\mathbb{E}}[|\hat{\mathbb{E}}_{(\tau\wedge T)+}[\varphi(X_{\tau+t_1}^x,\cdots,X_{\tau+t_m}^x)I_{\{\tau\leq T\}}]-\hat{\mathbb{E}}_{\tau+}[\varphi(X_{\tau+t_1}^x,\cdots,X_{\tau+t_m}^x)]|]\\
	&\ \ \rightarrow 0, \ \ \ \ \text{as}\ T\rightarrow \infty.
	\end{align*}
	Thus letting $T\rightarrow \infty$ in (\ref{23454464655}) yields (\ref{strong markov for sde}).
\end{proof}

Next we consider an extension of Theorem \ref{main theorem} in which the cylinder functions $\varphi$ is replaced by   (lower semi-) continuous functions $\widetilde{\varphi}$ depending on the whole paths of $G$-SDEs. It maybe useful in the following work.
\begin{theorem}\label{extended SDE strongmarkov1}
	Let $\varphi\in C_b(\Omega')$. Then
	\begin{equation}\label{12312342421453}\hat{\mathbb{E}}_{\tau+}[\varphi(X^x_{\tau+\cdot})]=\hat{\mathbb{E}}[ \varphi(X^y_\cdot)]_{y=X^x_{\tau}}.
	\end{equation}
\end{theorem}
The conditional expectation on the left-hand side of (\ref{12312342421453}) is meaningful by the following two lemmas.
\begin{lemma}\label{belong to L1tau lemma 2}
	Assume $\varphi\in C_{b}(\Omega')$ and there exists a constant  $\mu>0$ such that for some $T'>0$,
	\begin{equation}\label{3245323}|\varphi(\omega^1)-\varphi(\omega^2)|\leq \mu||\omega^1-\omega^2||_{C^n[0,T']}, \ \ \ \ \text{for each}\ \omega^1,\omega^2\in \Omega'.\end{equation}
	Then
	\begin{equation}
	\varphi(X^x_{\tau+\cdot})\in L^{1,\tau+}_G(\Omega).
	\end{equation}
\end{lemma}
\begin{remark}
\upshape{	Note that (\ref{3245323}) implies that $\varphi$ only depends on the path of $\omega\in \Omega'$ on $[0,T']$.}
\end{remark}
\begin{proof}
	As in the Step 2 of the proof of Lemma \ref{belong to L1tau lemma 1}, it suffices to
 suppose  that $\tau\leq T$ for some $T>0$.
Consider for each $m\in\mathbb{N}$ the function from $\mathbb{R}^{({m+1})\times n}$ to $\Omega'$ defined by
	$$
	\phi_m(x_0,x_1,x_2,\cdots,x_m)(t)=\sum_{k=0}^{m-1}\frac{(t^m_{k+1}-t)x_{k}+(t-t^m_k)x_{k+1}}{t^m_{k+1}-t^m_k}I_{[t^m_k,t^m_{k+1})}(t)+x_mI_{[t_m^m,\infty)},
	$$
	where $t^m_k=\frac{kT'}{m},k=0,1,\cdots,m$. Since $\varphi\circ\phi_m$ is a bounded, Lipschitz function from $\mathbb{R}^{({m+1})\times n}$ to $\mathbb{R}$, by Lemma \ref{belong to L1tau lemma 1}, we have
	$$
	\varphi(\phi_m(X^x_{\tau+t^m_0},X^x_{\tau+t^m_1},X^x_{\tau+t^m_2},\cdots,X^x_{\tau+t_m^m})) \in L^{1,\tau+}_G(\Omega).
	$$
	We employ the notation in the proof of
	Lemma \ref{belong to L1tau lemma 1} and  proceed similarly  to obtain   some constant $C>0$ depending on $\varphi$ such that
	\begin{align*}
	& \hat{\mathbb{E}}[|\varphi(\phi_m(X^x_{\tau+t^m_0},X^x_{\tau+t^m_1},\cdots,X^x_{\tau+t_m^m}))-\varphi(X^x_{\tau+\cdot})|] \\
	&\ \  \leq\hat{\mathbb{E}}[\sup_{0\leq t\leq T}|\varphi(\phi_m(X^x_{t+t^m_0},X^x_{t+t^m_1},\cdots,X^x_{t+t_m^m})-\varphi(X^x_{t+\cdot})|] \\
	&\ \ =\hat{\mathbb{E}}^x_2[\sup_{0\leq t\leq T}|\varphi(\phi_m(B'_{t+t^m_0},B'_{t+t^m_1},\cdots,B'_{t+t_m^m})-\varphi(B'_{t+\cdot})|]\\
	&\ \ \leq C\hat{\mathbb{E}}^x_2[\sup_{ s,s'\in \Lambda_{{m^{-1}}{T'},T+T'}}(|B'_s-B'_{s'}|\wedge 1)]\\
	&\ \ \rightarrow 0, \ \ \ \  \text{as}\ m\rightarrow\infty,
	\end{align*}
	This completes the proof.
\end{proof}
\begin{lemma}\label{belong to L1tau* lemma}
	Let $\varphi\in C_b(\Omega')$. Then
	\begin{equation}\varphi(X^x_{\tau+\cdot})\in L^{1,\tau+}_G(\Omega).
	\end{equation}
\end{lemma}
\begin{proof}
	Let
	$$
	\varphi_m(\omega):=\inf_{\omega'\in\Omega'}\{\varphi(\omega')+m||\omega-\omega'||_{C^n[0,m]}\},\ \ \ \ \text{for}\ \omega\in \Omega'.
	$$
	Then by Lemma 3.1 in Chap VI of \cite{P7}, $\varphi_m\in C_{b}(\Omega')$  satisfies
	\begin{itemize}
		\item [(i)]$|\varphi_m(\omega^1)-\varphi_m(\omega^2)|\leq m||\omega^1-\omega^2||_{C^n[0,m]},\ \text{for}\ \omega^1,\omega^2\in \Omega';$
		\item [(ii)]   $\varphi_m\uparrow \varphi$;
		\item [(iii)] $|\varphi_m|\leq C_\varphi$.
	\end{itemize}
Thus we have $\varphi_m(X^x_{\tau+\cdot})\in L^{1,\tau+}_G(\Omega)$ by Lemma \ref{belong to L1tau lemma 2}.

As discussed in the proof of Lemma \ref{belong to L1tau lemma 2}, it suffices to prove the result for $\tau\leq T$.
   Let $\hat{\mathbb{E}}^x_2$ and $c^x_2$ be defined as in the proof of
   Lemma \ref{belong to L1tau lemma 1}. We have
   \begin{align*}
 \hat{\mathbb{E}}[|\varphi_m(X^x_{\tau+\cdot})-\varphi(X^x_{\tau+\cdot})|]  &\leq\hat{\mathbb{E}}[\sup_{0\leq t\leq T}|\varphi_m(X^x_{t+\cdot})-\varphi(X^x_{t+\cdot})|]\\
   &=\hat{\mathbb{E}}^x_2[\sup_{0\leq t\leq T}|\varphi_m(B'_{t+\cdot})-\varphi(B'_{t+\cdot})|].
   \end{align*}
  Given any $\varepsilon>0$, since $c^x_2$ is tight on $\Omega'$, we can pick a compact set $K\subset \Omega'$ such that $c^x_2(K)<\varepsilon$. Note that $K\times [0,T]$ is still compact  and $(\omega,t)\mapsto \varphi_m(B'_{t+\cdot}),\varphi(B'_{t+\cdot})$ are  continuous functions such that $\varphi_m(B'_{t+\cdot})\uparrow\varphi(B'_{t+\cdot})$.  We have by Dini's theorem $$\varphi_m(B'_{t+\cdot})\uparrow\varphi(B'_{t+\cdot}) \ \ \ \  \text{uniformly on}\ \ K\times [0,T].$$
  Hence, we can choose $m$ large enough such that
   $$
  | \varphi_m(B'_{t+\cdot})-\varphi(B'_{t+\cdot})|\leq \varepsilon  \ \ \ \ \text{on}\ \ K\times [0,T].
   $$
   Then
   \begin{align*}
   &\hat{\mathbb{E}}^x_2[\sup_{0\leq t\leq T}|\varphi_m(B'_{t+\cdot})-\varphi(B'_{t+\cdot})|]\\
   &\ \ \leq \hat{\mathbb{E}}^x_2[\sup_{0\leq t\leq T}|\varphi_m(B'_{t+\cdot})-\varphi(B'_{t+\cdot})|I_K]+\hat{\mathbb{E}}^x_2[\sup_{0\leq t\leq T}|\varphi_m(B'_{t+\cdot})-\varphi(B'_{t+\cdot})|I_{K^c}]\\
   &\ \ \leq \varepsilon+2\varepsilon C_\varphi.
   \end{align*}
   Since $\varepsilon$ can be arbitrarily small, we obtain
   $$
    \hat{\mathbb{E}}[|\varphi_m(X^x_{\tau+\cdot})-\varphi(X^x_{\tau+\cdot})|] \rightarrow 0,\ \ \ \ \text{as}\ m\rightarrow\infty.
   $$
    This proves the lemma.
\end{proof}

\begin{proof}[Proof of Theorem \ref{extended SDE strongmarkov1}]
	\textit{Step 1.} Suppose $\tau\leq T$ for some $T>0$  and $\varphi\in C_{b}(\Omega')$ such that (\ref{3245323}) holds for some $T'>0$.
	
	For each $m\in\mathbb{N}$, we define $\phi_m$ as in  the proof of Lemma \ref{belong to L1tau lemma 2}.
	Then  Theorem \ref{main theorem} gives
	\begin{equation}\label{789000}
	\hat{\mathbb{E}}_{\tau+}[\varphi(\phi_m(X^x_{\tau+t^m_0},X^x_{\tau+t^m_1},X^x_{\tau+t^m_2},\cdots,X^x_{\tau+t_m^m}))]
	=\hat{\mathbb{E}}[\varphi(\phi_m(X^y_{t^m_0},X^y_{t^m_1},X^y_{t^m_2},\cdots,X^y_{t_m^m}))]_{y=X^x_{\tau}}.
	\end{equation}
	According to the proof of Lemma \ref{belong to L1tau lemma 2},
	$$
	\varphi(\phi_m(X^x_{\tau+t^m_0},X^x_{\tau+t^m_1},\cdots,X^x_{\tau+t_m^m}))\rightarrow \varphi(X^x_{\tau+\cdot}) \ \ \ \  \text{in}  \ \mathbb{L}^1,\ \text{as}\ m\rightarrow\infty.
	$$
	Consequently,
	\begin{equation*}\label{88888765656}
	\hat{\mathbb{E}}_{\tau+}[\varphi(\phi_m(X^x_{\tau+t^m_0},X^x_{\tau+t^m_1},\cdots,X^x_{\tau+t_m^m}))]\rightarrow \hat{\mathbb{E}}_{\tau+}[\varphi(X^x_{\tau+\cdot})]\ \ \ \ \text{in} \  \mathbb{L}^1,\ \text{as}\ m\rightarrow\infty.
	\end{equation*}

	It remains to consider the right side of (\ref{789000}).
	For any fixed $R>0$, by Kolmogorov's  criterion for tightness, the family $\mathcal{P}_R:=\bigcup_{y\in \overline{B_R(0)}}\{P\circ (X_{\cdot}^y)^{-1}:P\in\mathcal{P}\}$ is tight on $\Omega'$, where $B_R(0)$ is an open ball with center $0$ and radius $R$ in $\mathbb{R}^n$ and $\overline{B_R(0)}$ is its closure. We denote the corresponding sublinear expectation  by $\hat{\mathbb{E}}^R_2:=\sup_{P\in\mathcal{P},y\in \overline{B_R(0)} }E_{P\circ (X_{\cdot}^y)^{-1}}$.
	We may apply a similar analysis as in the proof of Lemma \ref{belong to L1tau lemma 2} to obtain for some constant $C$ depending on $\varphi$
	\begin{equation*}
	\begin{split}
	& \hat{\mathbb{E}}[|\varphi(\phi_m(X^y_{t^m_0},X^y_{t^m_1},\cdots,X^y_{t^m_m}))-\varphi(X^y_\cdot)|] \\
	& \ \ =\hat{\mathbb{E}}^y_2[|\varphi(\phi_m(B'_{t^m_0},B'_{t^m_1},\cdots,B'_{t^m_m}))-\varphi(B'_\cdot)|]\\
	& \ \ \leq \hat{\mathbb{E}}^R_2[|\varphi(\phi_m(B'_{t^m_0},B'_{t^m_1},\cdots,B'_{t^m_m}))-\varphi(B'_\cdot)|]\\
	&\ \ \leq C\hat{\mathbb{E}}^R_2[\sup_{ s,s'\in \Lambda_{{m^{-1}}{T'},T'}}(|B'_s-B'_{s'}|\wedge 1)]\\
	&\ \ \rightarrow 0,  \ \ \ \  \text{as} \ m\rightarrow \infty,  \   \text{for any}\ y \in \overline{B_R(0)},
	\end{split}
	\end{equation*}
where $\hat{\mathbb{E}}^y_2:=\sup_{P\in\mathcal{P}}E_{P\circ (X_{\cdot}^y)^{-1}}$.
That is,
	\begin{equation}\label{3324325343}
\hat{\mathbb{E}}[|\varphi(\phi_m(X^y_{t^m_0},X^y_{t^m_1},\cdots,X^y_{t^m_m}))-\varphi(X^y_\cdot)|]\rightarrow 0,  \ \ \ \ \text{as}\ m\rightarrow \infty, \ \text{uniformly for }\  y\in \overline{B_R(0)}.
	\end{equation}
	
For any fixed $\varepsilon>0$, we can first choose $R$ large enough such that  by  Lemma \ref{GSDE} (\ref{SDE sup control})
	$$
	c(\{|X^x_\tau|>R\})\leq \frac{\hat{\mathbb{E}}[|X^x_\tau|]}{R}\leq \frac{\hat{\mathbb{E}}[\sup_{t\in [0,T]}|X^x_t|]}{R}\leq \varepsilon
	$$
	and then choose $m$ large enough such that by  (\ref{3324325343})
	$$
	\hat{\mathbb{E}}[|\varphi(\phi_m(X^y_{t^m_0},X^y_{t^m_1},\cdots,X^y_{t^m_m}))-\varphi(X^y_\cdot)|]\leq \varepsilon,\ \ \ \ \text{for all}\ y\in \overline{B_R(0)}.
	$$
	Thus we have
	\begin{align*}
	&\hat{\mathbb{E}}[|\hat{\mathbb{E}}[\varphi(\phi_m(X^y_{t^m_0},X^y_{t^m_1},X^y_{t^m_2},\cdots,X^y_{t^m_m}))]_{y=X^x_{\tau}}-\hat{\mathbb{E}}[ \varphi(X^y_\cdot)]_{y=X^x_{\tau}}|]\\
	&\ \ \leq \hat{\mathbb{E}}[|\hat{\mathbb{E}}[\varphi(\phi_m(X^y_{t^m_0},X^y_{t^m_1},X^y_{t^m_2},\cdots,X^y_{t^m_m}))]_{y=X^x_{\tau}}-\hat{\mathbb{E}}[ \varphi(X^y_\cdot)]_{y=X^x_{\tau}}|I_{\{|X^x_\tau|\leq R\}}]+2C_\varphi c(\{|X^x_\tau|>R\})\\
	&\ \ \leq \varepsilon+2C_\varphi\varepsilon,
	\end{align*}
	which implies
	$$
	\hat{\mathbb{E}}[|\hat{\mathbb{E}}[\varphi(\phi_m(X^y_{t^m_0},X^y_{t^m_1},X^y_{t^m_2},\cdots,X^y_{t^m_m}))]_{y=X^x_{\tau}}-\hat{\mathbb{E}}[ \varphi(X^y_\cdot)]_{y=X^x_{\tau}}|]\rightarrow 0,  \ \ \ \  \text{as} \ m\rightarrow\infty.
	$$
	Therefore, letting $m\rightarrow\infty$ in  (\ref{789000}), we obtain
	$$\hat{\mathbb{E}}_{\tau+}[\varphi(X^x_{\tau+\cdot})]=\hat{\mathbb{E}}[ \varphi(X^y_\cdot)]_{y=X^x_{\tau}}.$$
	
	\textit{Step 2.}
	Assume $\tau\leq T$ and $\varphi\in C_b(\Omega)$. Define $\varphi_m$ as in the proof of Lemma \ref{belong to L1tau* lemma}.
    According to Step 1,
	\begin{equation}\label{444444}
	\hat{\mathbb{E}}_{\tau+}[\varphi_m(X^x_{\tau+\cdot})]=\hat{\mathbb{E}}[ \varphi_m(X^y_\cdot)]_{y=X^x_{\tau}}.
	\end{equation}
	Letting $m\rightarrow \infty$,  from the proof of  Lemma \ref{belong to L1tau* lemma}, we obtain  that
	$$
	\hat{\mathbb{E}}_{\tau+}[\varphi(X^x_{\tau+\cdot})]=\hat{\mathbb{E}}[ \varphi(X^y_\cdot)]_{y=X^x_{\tau}},
	$$
	where the convergence of right-hand side is obtained  by  a similar analysis as in  Step 1 and  the proof of Lemma \ref{belong to L1tau* lemma}.
	
	\textit{Step 3.} We proceed as in the last paragraph of the proof of Theorem \ref{main theorem} to obtain the result for the general case that $\tau$ is an optional time and $\varphi\in C_{b}(\Omega)$.
\end{proof}
\begin{corollary}\label{strong markov for lower semicontinuous function}
	Let $\varphi$ be lower semi-continuous on $\Omega'$ and bounded from below, i.e., $\varphi\geq c$ for some constant $c$. Then $\varphi(X^x_{\tau+\cdot}) \in L^{1,\tau+,*}_G(\Omega)$ and
	$$\hat{\mathbb{E}}_{\tau+}[ \varphi(X^x_{\tau+\cdot})]=\hat{\mathbb{E}}[ \varphi(X^y_{\cdot})]_{y=X^x_{\tau}}.$$
\end{corollary}
\begin{proof}
	We pick a sequence $\varphi_m\in C_b(\Omega')$ such that $\varphi_m\uparrow \varphi$. Then the conclusion follows from Theorem \ref{extended SDE strongmarkov1}, Lemma \ref{upward mct for rv} and  Proposition \ref{Etau proposition on L1tau*} (iv).
\end{proof}

Assuming $n=d$, $x= 0$, $b=h_{ij}=0$, $\sigma:=(\sigma_1,\cdots,\sigma_d)=I_{d\times d}$ in Corollary \ref{strong markov for lower semicontinuous function}, we immediately have the strong Markov property for $G$-Brownian motion.
\begin{corollary}\label{extended BM strongmarkov1}
	Let $\varphi$ be lower semi-continuous, bounded from below  on $\Omega$ and  $\tau$ be an optional time. Then
	\begin{equation}\label{43543}\hat{\mathbb{E}}_{\tau+}[\varphi(B_{\tau+\cdot})]=\hat{\mathbb{E}}[\varphi(B_\cdot^y)]_{y=B_{\tau}},
	\end{equation}
	where $B_t^y:=y+B_t,\ t\geq 0$ for $y\in\mathbb{R}^d$. In particular,
	for each $\phi\in C_{b.Lip}(\mathbb{R}^{m\times d})$ and $0\leq t_1\leq \cdots\leq t_m<\infty$,
	\begin{equation*}
	\hat{\mathbb{E}}_{\tau+}[\phi(B_{\tau+t_1},\cdots,B_{\tau+t_m})]=\hat{\mathbb{E}}[\phi(B^y_{t_1},\cdots,B^y_{t_m})]_{y=B_{\tau}}.
	\end{equation*}
\end{corollary}
The following result says that  $G$-Brownian motion starts afresh at an optional time, i.e., $\overline{B}_t:=(B_{\tau+t}-B_{\tau})_{t\geq 0}$ is still a $G$-Brownian motion.
\begin{corollary}\label{refresh of G-Brownian motion}
	Let $\tau,\varphi$ be assumed as in the above Corollary. Then
		\begin{equation}\hat{\mathbb{E}}_{\tau+}[\varphi(B_{\tau+\cdot}-B_{\tau})]=\hat{\mathbb{E}}[\varphi(B_{\tau+\cdot}-B_{\tau})]=\hat{\mathbb{E}}[\varphi(B_\cdot)].
	\end{equation}
In particular, for each $\phi\in C_{b.Lip}(\mathbb{R}^{m\times d})$, $0\leq t_1\leq\cdots\leq t_m<+\infty$, $m\in\mathbb{N}$, we have
\begin{align*}
\hat{\mathbb{E}}_{\tau+}[\phi(B_{\tau+t_1}-B_{\tau},\cdots,B_{\tau+t_m}-B_{\tau})]=\hat{\mathbb{E}}[\phi(B_{\tau+t_1}-B_{\tau},\cdots,B_{\tau+t_m}-B_{\tau})]=\hat{\mathbb{E}}[\phi(B_{t_1},\cdots,B_{t_m})].
\end{align*}	
\end{corollary}
\begin{proof}
	We only need to prove the first one, which implies the second one as a special case. Setting ${\tilde{\varphi}}(\omega):=\varphi((\omega_t-\omega_0)_{t\geq 0})$ in (\ref{43543}), we have
	$$
	\hat{\mathbb{E}}_{\tau+}[\varphi(B_{\tau+\cdot}-B_{\tau})]=\hat{\mathbb{E}}[\varphi(B_{\cdot})].
	$$
	Taking expectation on both sides, by
	 Proposition \ref{Etau welldefined}, we then obtain
	$$	\hat{\mathbb{E}}[\varphi(B_{\tau+\cdot}-B_{\tau})]=\hat{\mathbb{E}}[\varphi(B_{\cdot})].$$
\end{proof}

\section{An application}
Let $(B_t)_{t\geq 0}$ be a 1-dimensional $G$-Brownian motion such that $\underline{\sigma}^2:=-\mathbb{\hat{E}}[-B^2_1]>0$ (non-degeneracy). Let $a\in\mathbb{R}$ be given.  For each $\omega\in\Omega$, define the  level set
\begin{equation}
\mathcal{L}_\omega(a):=\{t\geq 0: B_t(\omega)=a\}.
\end{equation}
It is proved in \cite{WZ} that $\mathcal{L}_\omega(a)$ is q.s. closed and has zero Lebesgue measure. Using the strong Markov property for $G$-Brownian motion, we can obtain the following theorem.
\begin{theorem}\label{no isolate point theorem for G-Bm}For q.s. $\omega\in\Omega$, the level set $\mathcal{L}_\omega(a)$ has no isolated point in $[0,\infty)$.
\end{theorem}
To prove Theorem 5.1, we need the following two lemmas.
\begin{lemma}\label{Bm change sign lemma}
	For q.s. $\omega$,	$G$-Brownian motion $(B_t)_{t\geq 0}$ changes sign infinitely many times in  $[0,\varepsilon]$, for any $\varepsilon>0$.
\end{lemma}
\begin{proof}
	Define $\tau_1:=\inf\{t> 0:B_t>0\}.$ Then $\tau_1$ is an optional time by Lemma 7.6 in Chap 7 of \cite{Ka}. Let $P\in\mathcal{P}$ and $t\geq 0$ be given. Since $B$ is a martingale,
	 we can apply  the classical optional sampling theorem to  obtain
	${E}_P[-B_{\tau_1\wedge t}]=0$. Thus $\mathbb{\hat{E}}[-B_{\tau_1\wedge t}]=0$. Noting that $-B_{\tau_1\wedge t}\geq 0$, we then have $-B_{\tau_1 \wedge t}=0$ q.s., i.e., $B_{\tau_1 \wedge t}=0\ \text{q.s.}$ Similar analysis for $-B$ shows $B_{\tau_2 \wedge t}=0$ q.s.,  for $\tau_2:=\{t>0:B_t<0\}$. Therefore, $B_{\tau_0 \wedge t}=0$ q.s., for $\tau_0:=\tau_1\vee\tau_2$. This implies $B_{\tau_0 \wedge t}=0\  \text{for each}\ t\geq 0,\ \text{q.s.}$
	
	Applying Proposition 1.13 in Chap IV of \cite{Marc} under each $P\in\mathcal{P}$, we then have $\langle B\rangle_{\tau_0 \wedge t}=0\ \text{for each}\ t\geq 0,\ \text{q.s.}$ But from Corollary 5.4 in Chap III of \cite{P7} that ${\langle B\rangle_{t+s}-\langle B\rangle_{t}}\geq  \underline{\sigma}^2s>0$ for each $s> 0$, we must have  $\tau_0=0$ q.s. Hence,
	$\tau_1=0$ and $\tau_2=0,\ \text{q.s.},$ which imply the desired result.
\end{proof}
\begin{lemma}\label{Bm unbound lemma}
	We have
	\begin{equation}
	\sup_{0\leq t<\infty}B_t=+\infty  \ \ \text{and}\ \  \inf_{0\leq t<\infty}B_t=-\infty, \ \ \ \ \text{q.s.}
	\end{equation}
\end{lemma}
\begin{proof}
	We only prove the first equality, from which  the second one follows by the symmetry of $G$-Brownian motion.
	
	Define $\tau_t=\inf\{s\geq 0:\langle B\rangle_s>t\}$. Under each $P\in\mathcal{P}$, $B$ is a martingale. Then by Theorem 1.6 in Chap V of \cite{Marc}, $(B_{\tau_t})_{t\geq 0}$ is a classical Brownian motion. Applying Lemma 3.6 in  Chap I of \cite{Rog}, we have
	$$
	\sup_{0\leq t<\infty}B_{\tau_t}=+\infty\ \ \ \ P\text{-a.s.}
	$$
	Since $\{\tau_t:t\in [0,\infty)\}=[0,\infty),$ we then obtain
	$$
	\sup_{0\leq t<\infty}B_{t}=+\infty\ \ \ \ P\text{-a.s.}
	$$
	Therefore,
	$$
	\sup_{0\leq t<\infty}B_{t}=+\infty\ \ \ \  \text{q.s.}
	$$
\end{proof}
\begin{remark}\label{level set unbound remark}
	\upshape{This lemma implies that $\mathcal{L}_\omega(a)$ is q.s. unbounded.}
\end{remark}
\begin{proof}[Proof of Theorem \ref{no isolate point theorem for G-Bm}]
	Let $t\geq 0$. Define the optional time after  $t$
	$$
	\tau_t=\inf\{s>t:B_s=a\}.
	$$
	By Lemma \ref{Bm unbound lemma} (see also Remark \ref{level set unbound remark}), $\tau_t$ is q.s. finite.
	Now we are going to show that
	\begin{equation}
	\label{332432353532}
	\tau_{\tau_t}=\inf\{{s>\tau_t}:B_s=a\}={\tau_t}\ \ \ \  \text{q.s.}
	\end{equation}
	
	For any $n\geq 1$, since $\tau_t\wedge n$ satisfies (H3), then Corollary \ref{refresh of G-Brownian motion} implies that $(B_{\tau_t\wedge n+s}-B_{\tau_t\wedge n})_{s\geq 0}$ is still a $G$-Brownian motion.
	Hence,  by Lemma \ref{Bm change sign lemma}, there exists a set $\Omega_n\subset \Omega$ such that $c(\Omega_n^c)=0$ and on $\Omega_n$, $(B_{\tau_t\wedge n+s}-B_{\tau_t\wedge n})_{s\geq 0}$ changes its sign infinitely many times on any $[0,\varepsilon]$.
	
	Let
	$$
	\Omega_0:=\bigcup_{n=1}^\infty(\Omega_n\cap\{\tau_t\leq n\}).
	$$
	For any $P\in\mathcal{P}$, we have
	$$P(\Omega_0^c)=P(\bigcap_{n=1}^\infty (\Omega_n^c\cup\{\tau_t> n\}))\leq P(\Omega_n^c\cup\{\tau_t> n\})=P(\{\tau_t> n\})\rightarrow P(\{\tau_t=\infty\})=0, \ \ \ \  \text{as} \ n\rightarrow\infty.
	$$
    Thus $$
	c(\Omega_0^c)=0.
	$$
	For any fixed $\omega\in \Omega_0$, there exists an $n$ such that $\omega\in \Omega_n\cap\{\tau_t\leq n\}$. Since $\tau_t(\omega)\wedge n=\tau_t(\omega)$, then $((B_{\tau_t+s}-B_{\tau_t})(\omega))_{s\geq 0}$ changes its sign infinitely many times on any $[0,\varepsilon]$. Therefore,
	$$\tau_{\tau_t}(\omega)={\tau_t}(\omega),$$
	which proves (\ref{332432353532}).
	
	Note that, for any fixed $p<q$,
	$$\Lambda_{p,q}:=\{\omega\in\Omega: \ \text{there is only one}\ s\in (p,q) \ \text{such that}\ B_s(\omega)=a\}\subset \{\omega\in\Omega:\tau_p<q, \tau_{\tau_p}\geq q \}.$$
	We must have $c(\Lambda_{p,q})=0$. Thus the set $$\{\omega\in \Omega:\ \mathcal{L}_\omega(a)\ \text{has isolated point}\}=\bigcup_{0\leq p<q;\ p,q\in Q}\Lambda_{p,q}$$
	is a zero capacity set.
\end{proof}

\end{document}